\documentclass[12pt]{amsart}

\usepackage[left=3.5cm, right=3.5cm, top=3cm]{geometry}
\usepackage{amssymb,amsmath,amsthm,enumitem,MnSymbol,array}
\usepackage{graphicx}
\graphicspath{{./images/} }
\usepackage{xcolor}
\usepackage{hyperref}
\usepackage{mathbbol}
\usepackage[font=small]{caption}
\hypersetup{
  pdftitle={Quadrangles entangled in conic nets},
  pdfauthor={Bruce Olberding and Elaine A. Walker},
  pdfsubject={Quadrangles, nets of conics, bisector fields, and entanglement},
  pdfkeywords={Planar four-point configuration, quadrangle, pencil of conics, net of conics, bisector field, entanglement, nine-point conic}
}

\def\AA{\mathsf{A}}
\def\BB{\mathsf{B}}
\def\CC{\mathsf{C}}

\def\P{{\mathbb{P}}}

\begin{document}

\newtheorem{theorem}{Theorem}[section]
\newtheorem{lemma}[theorem]{Lemma}
\newtheorem{proposition}[theorem]{Proposition}
\newtheorem{corollary}[theorem]{Corollary}
\newtheorem{problem}[theorem]{Problem}
\newtheorem{construction}[theorem]{Construction}

\theoremstyle{definition}
\newtheorem{defi}[theorem]{Definitions}
\newtheorem{definition}[theorem]{Definition}
\newtheorem{notation}[theorem]{Notation}
\newtheorem{remark}[theorem]{Remark}
\newtheorem{example}[theorem]{Example}
\newtheorem{question}[theorem]{Question}
\newtheorem{comment}[theorem]{Comment}
\newtheorem{comments}[theorem]{Comments}

\newtheorem{discussion}[theorem]{Discussion}

\renewcommand{\thedefi}{}

\long\def\alert#1{\smallskip{\hskip\parindent\vrule%
\vbox{\advance\hsize-2\parindent\hrule\smallskip\parindent.4\parindent%
\narrower\noindent#1\smallskip\hrule}\vrule\hfill}\smallskip}

\def\ff{\frak}
\def\tf{torsion-free}
\def\Spec{\mbox{\rm Spec }}
\def\Proj{\mbox{\rm Proj }}
\def\hgt{\mbox{\rm ht }}
\def\type{\mbox{type}}
\def\Hom{\mbox{Hom}}
\def\rank{\mbox{rank}}
\def\Ext{\mbox{Ext}}
\def\Tor{\mbox{Tor}}
\def\ker{\mbox{Ker }}
\def\Max{\mbox{\rm Max}}
\def\End{\mbox{\rm End}}
\def\xpd{\mbox{\rm xpd}}
\def\Ass{\mbox{\rm Ass}}
\def\emdim{\mbox{\rm emdim}}
\def\epd{\mbox{\rm epd}}
\def\repd{\mbox{\rm rpd}}
\def\ord{\mbox{\rm ord}}

\def\htt{\mbox{\rm ht}}

\def\EE{{\mathcal E}}
\def\FF{{\mathcal F}}
\def\GG{{\mathcal G}}
\def\HH{{\mathcal H}}
\def\II{{\mathcal I}}
\def\LL{{\mathcal L}}
\def\MM{{\mathcal M}}
\def\PP{{\mathsf P}}

\def\k{\mathbb{k}}

\title[Quadrangles entangled in conic nets]{Quadrangles entangled in conic nets}
\author{Bruce Olberding} 
\address{Department of Mathematical Sciences, New Mexico State University,\\ Las Cruces, NM 88003-8001}

\email{bruce@nmsu.edu}

\author{Elaine A.~Walker}
\address{Las Cruces, NM}

\email{miselaineeous@yahoo.com}

\begin{abstract}
A net of conics can remain unchanged as the four-point configuration generating it varies. We study this phenomenon over a field $\k$ of characteristic different from $2$. We allow degenerate quadrangles, consisting of four not necessarily distinct points and six joining lines, provided that three sides form a nondegenerate triangle. The net of a quadrangle is obtained by adjoining constants to the defining polynomials in the pencil generated by its pairs of opposite sides, viewed as conics. We call two quadrangles entangled when they generate the same net. We prove that entanglement is equivalent to the apparently asymmetric condition that every side of one quadrangle bisects the other quadrangle at that side's midpoint. The line-pair members of the net form a bisector field that determines the entire net, so entanglement is also equivalent to equality of bisector fields. Finally, we prove that the locus of centers of the central conics in the net, together with one nonparallel line-pair member, determines the net.

\end{abstract}



\date{\today}

\maketitle

\section{Introduction}

Let $\k$ be a field of characteristic other than $2$. 
Given four points in general position in the  plane $\k^2$, there is a natural way to associate both a pencil and a net of conics to the configuration. The six lines through these four points can be grouped into three pairs of opposite sides  that can be viewed as reducible conics belonging to a common pencil of conics, namely the pencil of conics through the four points. 
 By adjoining constants to the defining polynomials of this pencil, we obtain  a three-dimensional vector space  whose projectivization is a net   that  gathers the level conics of the pencil into a single projective linear system. While the pencil recovers the four points as its base locus, the net of its level conics does not, and the four-point configuration can vary while this larger linear system remains fixed. We call two such configurations {\it entangled} if they determine the same net.  Their common net is a projective plane in the space of conics, and each four-point configuration's pencil is a projective line in that plane. The points on this line represent  the conics in the pencil. We study the geometry shared by entangled four-point configurations and ask what data determines their common net.

 Four-point configurations and their degenerations arise naturally in shape spaces and moduli problems, as in \cite{GonLop1,GonLop2,Kendall,OWManifold}, and in these contexts degenerate configurations are important as a boundary case. 
 In our approach in \cite{OWManifold} we work with quadrangles in place of four-point configurations since quadrangles, which consist of both the original four points and the six lines joining any two of them, retain the data of the lines between points and thus {allow a finer distinction among degeneracies}. We allow degenerate quadrangles in our study  \cite{OWManifold} of moduli spaces of quadrangles and do so here too, 
but as in the moduli setting we impose  a noncollapse condition   requiring that three of the six sides of the quadrangle form a nondegenerate triangle. A sequel to this article will show how entanglement is useful in the analysis of moduli spaces of quadrangles. 

Even if the quadrangle $Q$ is degenerate, we show we can associate a net to   $Q$  as above in a well-defined way. 
 The reducible conics in the net form a maximal self-bisecting family $\mathcal B(Q)$ of line pairs, in the sense that every component line bisects every pair in the family at the same bisection point. We call $\mathcal B(Q)$ the {\it bisector field} of $Q$ and show it is indeed a bisector field in the sense of 
 \cite{OWQuad,OWPencil,OWProj}; see Figure~\ref{cardioid}. We prove that the lines in the pairs in $\mathcal B(Q)$ are   the bisectors of $Q$ and that which bisectors occur together in its line pairs can be recognized directly from the six sides. More generally, Theorem~\ref{conic net criterion} shows that a conic belongs to the net of $Q$ exactly when every side of $Q$ bisects it at the side midpoint. 
 \begin{figure}[ht]
     \begin{center}
\includegraphics[width=.75\textwidth]{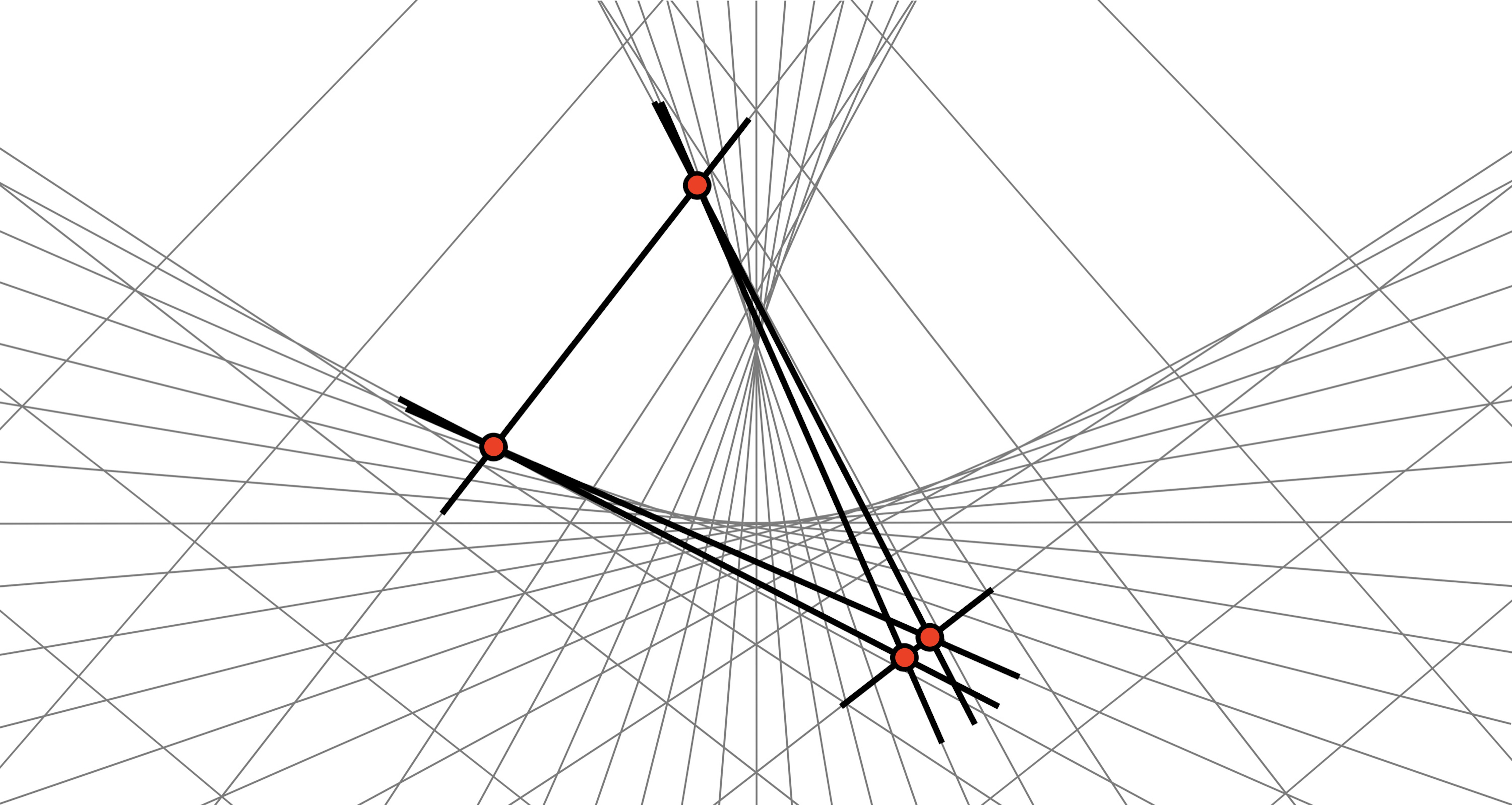}
     \end{center}
\caption{A quadrangle in its bisector field. The four red points and six thick black joining lines form the quadrangle. The light lines belong to pairs in its bisector field.}
\label{cardioid}
\end{figure}
 
We prove that two quadrangles $Q_1$ and $Q_2$ are entangled if and only if $\mathcal B(Q_1)=\mathcal B(Q_2)$, so that the line pairs in the net alone are decisive for entanglement. We also give a particularly transparent criterion for entanglement: $Q_1$ and $Q_2$ are entangled if and only if 
every side of $Q_1$ bisects $Q_2$ at its side midpoint. 
Thus, although this condition is imposed in only one direction, it is symmetric; see Figure~\ref{entangled quads}.
\begin{figure}[ht]
     \begin{center}
\includegraphics[width=.95\textwidth]{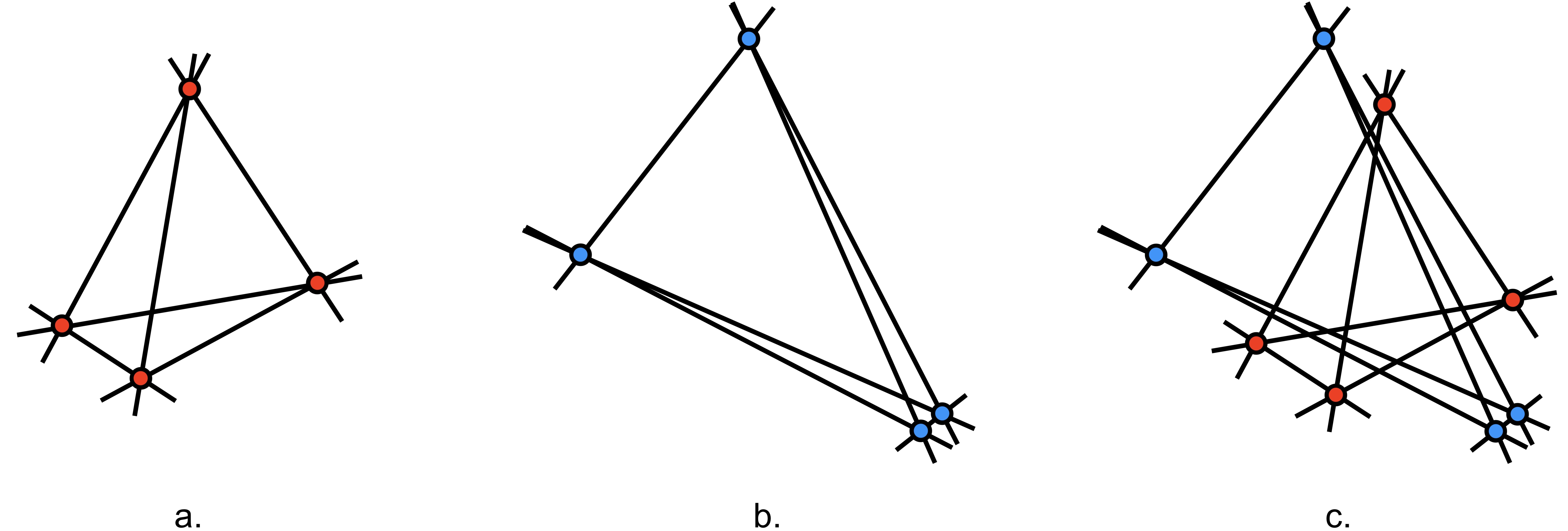}
     \end{center}
\caption{After the translation shown in (c), each side of the first quadrangle bisects the three pairs of opposite sides of the second at that side's midpoint. Corollary~\ref{entanglement bisection} implies that the  quadrangles are entangled and that the same bisection relation holds with the quadrangles interchanged.}
\label{entangled quads}
\end{figure}

 We next ask how much information is needed to recover the common net. For this purpose, we introduce a conic arising from the first derivatives of the pencil. If $F,G$ form a basis of its pencil space, then the Jacobian $$J(F,G)=(\partial_XF)(\partial_YG)-(\partial_YF)(\partial_XG)$$
defines a conic independent of the chosen basis. We prove that its $\k$-points are simultaneously the centers of the central conics in the pencil and the finite bisection points of the bisectors of $Q$. We call this set the {\it locus of $Q$} and the conic defined by $J(F,G)$ its {\it locus conic}. For a nondegenerate quadrangle, the locus conic is the classical nine-point conic. We determine its possible degenerations and show that it is centered at the centroid of $Q$.

Although the locus conic is constructed from the net rather than chosen as one of its members, it nearly determines the net.  Theorem~\ref{locus determines} shows that two quadrangles are entangled if and only if they have the same locus and share one nonparallel bisector pair. Equivalently, the locus conic together with one nonparallel line-pair  determines the entire quadrangle net. 

\smallskip

{\it Notation.} Throughout the article, $\k$ represents an arbitrary field of characteristic other than $2$, and $\P^n$ represents $n$-dimensional projective space over $\k$, with points given in homogeneous coordinates. Our focus is on configurations of lines and points in the plane $\k^2$. Unless explicitly stated otherwise, a line means an affine line in $\k^2$, and its points are understood to be its $\k$-points. Projective closures and points on the line at infinity will always be indicated explicitly.

\section{Quadrangles}\label{quadrangles section}

We loosen the classical definition of a quadrangle, which requires vertices to be in general position, to permit  degenerate quadrangles also because they arise naturally as boundary cases in moduli spaces, as in \cite{OWManifold}. However, we do not allow the quadrangles to degenerate too far. We require three of the six sides to form a nondegenerate triangle. This core of non-degeneracy is crucial for the principal results on bisectors and entanglement in this article, as well as the geometric results on moduli spaces of real quadrangles  in \cite{OWManifold}.

 \begin{definition} \label{quadrangle def} A {\it quadrangle} $Q$ is a planar configuration of four not necessarily distinct points $v_1,\ldots,v_4$, called {\it vertices}, and six indexed lines $L_{ij}$, $1 \leq i<j \leq 4$, called {\it sides}, such that each line $L_{ij}$ passes through $v_i$ and $v_j$ and  three of the sides form a nondegenerate triangle. Sides and vertices need not be distinct. 
  The {\it side midpoint} of the side $L_{ij}$ is the midpoint of $v_i$ and $v_j$, and the {\it centroid} of $Q$ is the centroid of the four vertices. Two sides $L_{ij}$ and $L_{i'j'}$ are called {\it opposite sides} if $\{i,j\} \cap \{i',j'\} = \emptyset$; otherwise, they are {\it adjacent}. A quadrangle is {\it nondegenerate} if its four  vertices are distinct and no three are collinear; otherwise it is {\it degenerate}.

 \end{definition}

Throughout the article, quadrangles are considered unlabeled.  The definition can be unpacked to make clearer what the degenerate quadrangles are, which we do in Proposition~\ref{unpack} and illustrate in Figure~\ref{degenerate quadrangles figure}. By a {\it line pair} we mean two not-necessarily-distinct lines $L$ and $L'$ considered without order, and we write this as $\{L,L'\}$, even though the ``pair'' may consist of only one line.  A variation on the proof of the proposition appears also in the real number setting of  \cite{OWManifold}, and we include it here for completeness.

 \begin{proposition} \label{unpack} \label{Ent}  Let $\AA,\BB,\CC$ be three pairs of opposite sides of a quadrangle $Q$. The quadrangle $Q$ is degenerate if and only if, up to renaming the  line pairs, one of the following conditions holds.
 \begin{enumerate}
\item[{\rm (a)}] $\AA=\CC$, the two lines in $\AA$ are distinct, and  neither line of $\AA$ belongs to $\BB$;

\item[{\rm (b)}] all three pairs $\AA,\BB,\CC$ share a line; or

\item[{\rm (c)}]  $\AA=\CC$, the two lines in $\AA$ are the same and this line does not belong to $\BB$.
 \end{enumerate}
 
 \end{proposition}
 
  \begin{figure}[ht] 
     \begin{center}
\includegraphics[width=.95\textwidth]{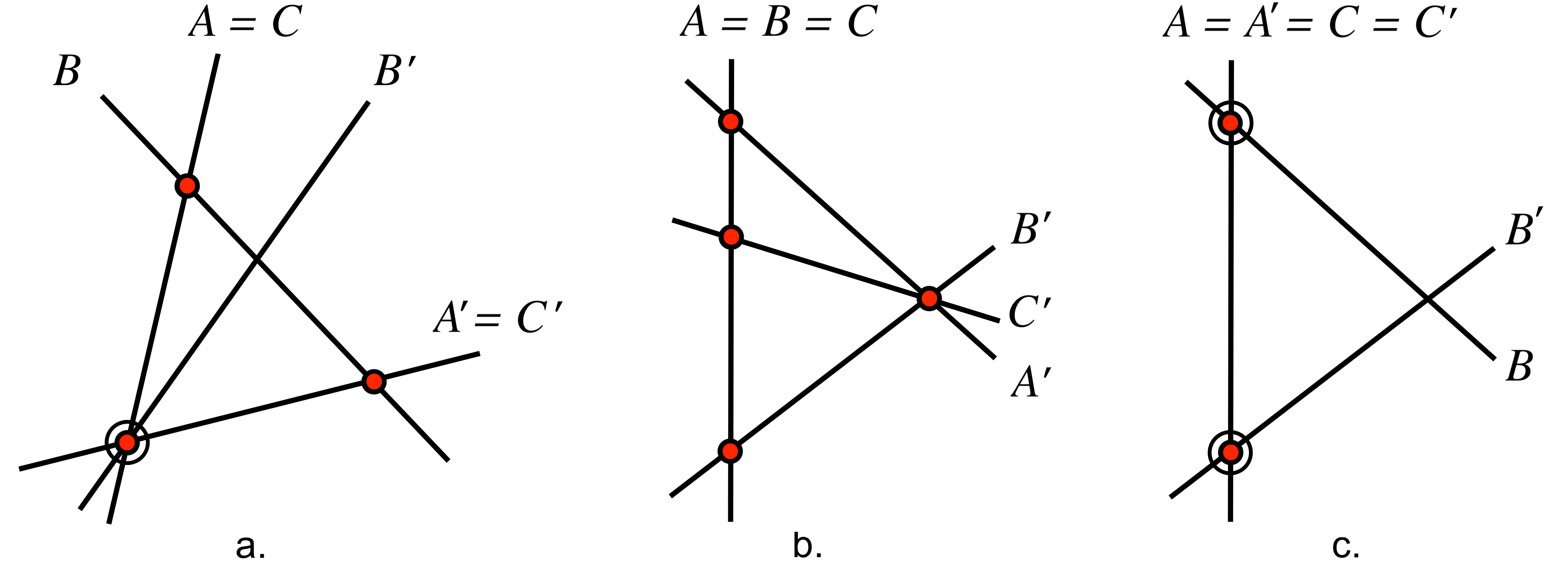}
 \end{center}
\caption{ Representative quadrangles for the cases in Proposition~\ref{unpack}: (a) a double-side quadrangle, (b) a triple-side quadrangle, and (c) a quadruple-side quadrangle.   Opposite-side pairs are labeled $A,A'$, $B,B'$, and $C,C'$, and double vertices are circled.}
{\label{degenerate quadrangles figure}}
\end{figure}

 \begin{proof} A nondegenerate quadrangle has six distinct lines, whereas each of conditions~(a), (b), and~(c) has non-distinct sides, so each condition in the proposition implies that the quadrangle is degenerate.
Conversely, suppose $Q$ is degenerate.   If all four vertices are collinear, the definition of  quadrangle excludes  four distinct vertices, at least three coincident vertices, and one double vertex together with two other distinct vertices, so in this case the vertices consist of two distinct double vertices. If the four vertices are not collinear, then neither three coincident vertices nor two double vertices can occur, since either configuration would be collinear. Hence either there is one double vertex and the other two vertices are not collinear with it, or the four vertices are distinct and exactly three are collinear. After   relabeling the vertices and the sides, we consider these three cases.

{\textsc{Case 1:}} {\it There are two distinct double vertices.}  Write $v_1=v_2=p$ and $v_3=v_4=q$, where $p\ne q$, and let $T$ be the line through $p$ and $q$. Then $L_{13}=L_{14}=L_{23}=L_{24}=T.$ Hence the pairs of opposite sides $\{L_{13},L_{24}\}$ and $\{L_{14},L_{23}\}$ both have $T$ as each of their two sides.  The remaining sides are $L_{12}$ and $L_{34}$. Since three sides form a nondegenerate triangle, neither of these lines is $T$, so~(c) holds.

{\textsc{Case 2:}} {\it There is exactly one double vertex, and the three distinct vertices are not collinear.}  Write $v_1=v_2=p$, $v_3=q$, and $v_4=r$, where $p,q,r$ are noncollinear. Let $L,M,N$ be the lines $pq,pr,qr$, respectively, and set $S=L_{12}$. Then $$ \{L_{13},L_{24}\}=\{L,M\}=\{L_{14},L_{23}\}, \:\: \{L_{12},L_{34}\}=\{S,N\}.$$ Since $p,q,r$ are noncollinear, $N \not \in \{L,M\}$. If $S=L$ or $S=M$, all three pairs of opposite sides share that line, and~(b) holds. Otherwise, neither line in  $\{L,M\}$ occurs in $\{S,N\}$, and~(a) holds.

{\textsc{Case 3:}} {\it Three vertices are distinct and collinear and the fourth vertex is not collinear with these three.} After relabeling, suppose $v_1,v_2,v_3$ lie on a line $T$ and $v_4$ does not. Then $L_{12}=L_{13}=L_{23}=T.$ Each of the three pairs of opposite sides of $Q$ contains one of these sides equal to $T$, so~(b) holds.
 \end{proof}

\begin{definition}
 For a line occurring among the six sides of a quadrangle, its \emph{side multiplicity} is the number of times it occurs. A line of side multiplicity two, three, or four is  a \emph{double side}, \emph{triple side}, or \emph{quadruple side}, respectively. A degenerate quadrangle is called \emph{double-side}, \emph{triple-side}, or \emph{quadruple-side} according to whether its greatest side multiplicity is two, three, or four.
\end{definition}

Thus the quadrangles in cases~(a), (b), and~(c) of Proposition~\ref{unpack} are double-side, triple-side, and quadruple-side, respectively. 
We  consider also two familiar classes of quadrangles.

 \begin{definition} \label{trap def} A quadrangle is a
{\it trapezoid} if it has a pair of parallel opposite sides.  A quadrangle is a {\it parallelogram} if it has two pairs of parallel opposite sides, with neither pair parallel to the other.
  \end{definition}

  \section{The pencil and net of a quadrangle}\label{conic pencil section}
  
A nondegenerate quadrangle has four vertices in general position, and it is a classical theorem that the conics containing all four vertices form a pencil parameterized by the projective line $\P^1$; see  \cite[Theorem~5.13, p.~91]{BambergPenttila}. Because we also allow degenerate quadrangles, we must define our pencil a little more generally, though it will still be parameterized by $\P^1$. We establish some terminology.

\smallskip

{\it Conventions for conics.} If $0\ne H\in\k[X,Y]$ has degree at most~$2$, its degree-two homogenization is $$H^h(X,Y,Z)=Z^2H(X/Z,Y/Z).$$ Write $[H]$ for the class of nonzero scalar multiples of $H$. We call $[H]$ a {\it conic}, and we call $H$ a {\it defining polynomial} of $[H]$. If $C=[H]$, we write $$C(\k)=V(H)=\{p\in\k^2:H(p)=0\},$$
for the   $\k$-point set of $C$. Its projective realization is the conic $H^h=0$. The exceptional conic $[1]$, which will appear when dealing with the net of a quadrangle, is supported on the line at infinity with multiplicity two and has empty (affine)  $\k$-point set. The conic $[H]$ is {\it degenerate} if $H^h$ factors over the algebraic closure $\overline{\k}$ of $\k$ as a product of linear forms, and is {\it nondegenerate} otherwise.  If $H$ is nonconstant, a {\it center} of $[H]$ is a point $m\in\k^2$ such that $\nabla H(m)=0$. Such a conic is {\it central} if it has a center.

\smallskip

  We identify a line pair $\mathsf L=\{L,L'\}$ with the conic $[F_{\mathsf L}]$, where a \emph{line-pair polynomial} $F_{\mathsf L}$ is the product of linear polynomials defining $L$ and $L'$, with repetition when $L=L'$. Let  $\AA$ and $\BB$ be two distinct pairs of opposite sides of a quadrangle $Q$, and let $F_{\AA}(X,Y)$ and $F_{\BB}(X,Y)$ be line-pair polynomials for $\AA$ and $\BB$, respectively.

\begin{definition}\label{bf def}
The {\it pencil space} of the quadrangle $Q$ is the $\k$-vector space $${\mathcal V}(Q)=\{rF_\AA+sF_\BB:r,s\in\k\}.$$
The {\it pencil of $Q$} is the projective linear system
$${\mathcal P}(Q)=\P({\mathcal V}(Q))=\{[rF_\AA+sF_\BB]:[r:s]\in\P^1\},$$
and the {\it net of $Q$} is the projective linear system $$\mathcal N(Q)=\mathbb P\bigl(\mathcal V(Q)\oplus\k\mathbin{\cdot}1\bigr) =\{[rF_\AA+sF_\BB+t]:[r:s:t]\in\P^2\}. $$ We denote by $\mathcal B(Q)$ the line pairs in $\mathcal N(Q)$, i.e.,  the set of members of $\mathcal N(Q)$ represented by products of two nonconstant linear polynomials over $\k$. 
 \end{definition}

The next proposition shows that these definitions are independent of the choice of $\AA$ and $\BB$. Sections~\ref{bisector field section} and~\ref{bisector pairs section} {will interpret} $\mathcal B(Q)$ {geometrically}.



\begin{proposition}\label{pencil independent}
For a quadrangle $Q$,  $\mathcal V(Q)$ is a two-dimensional vector space whose {nonzero} elements have degree $2$, 
 the pencil $\mathcal P(Q)$ is a projective line, and the net $\mathcal N(Q)$ is a projective plane. 
 Moreover,  $\mathcal V(Q)$, $\mathcal P(Q)$, $\mathcal N(Q)$ and $\mathcal B(Q)$ are independent of the choice of two distinct pairs of opposite sides. 
\end{proposition}

\begin{proof}
If  $Q$ is nondegenerate, then by
\cite[Theorem~4.25]{DragovicRadnovic}, the projective conics through the four
vertices of $Q$ form a pencil. Since the three pairs of opposite sides are distinct
members of this pencil, the polynomials arising from any two of them form a
basis of its two-dimensional vector space. Thus, if $Q$ is nondegenerate, $\mathcal V(Q)$ is independent
of the chosen two line pairs.
Suppose now that $Q$ is degenerate. In cases~(a) and~(c) of Proposition~\ref{unpack},  there are only two distinct pairs of opposite sides, so there is nothing to prove. In~(b), let $T$ be the triple side, let $v$ be the vertex off $T$, and let $F_T$ be a linear polynomial defining $T$. The space  $$U_v=\{G\in\k[X,Y]_{\leq1}:G(v)=0\}$$
 has dimension~$2$. All three  opposite-side polynomials lie in the two-dimensional space $F_TU_v$.  Any two such polynomials arising from distinct pairs of opposite sides are linear independent and hence form a basis of $F_TU_v$. Therefore, every choice of two distinct pairs of opposite sides {determines} the same pencil space in case~(b).
This proves that in all cases, $\mathcal V(Q)$ is independent of this choice of distinct pairs of opposite sides.
The definitions of $\mathcal P(Q)$ and $\mathcal N(Q)$ now show that the pencil and net are independent as well. Since $\mathcal B(Q)$ is defined as the family of line-pair members of $\mathcal N(Q)$, it is independent of the chosen line pairs as well.

{It remains to prove the assertions about dimension and degree. Choose
two distinct pairs of opposite sides $\AA$ and $\BB$. If the homogeneous
quadratic parts $(F_{\AA})_2$ and $(F_{\BB})_2$ were linear dependent, then the
two line pairs would have the same directions, counted with repetition, but since $Q$ is a quadrangle, every
line in $\AA$ shares a vertex with every line in $\BB$, and parallel  
lines through a common point coincide, this implies  $\AA=\BB$, a
contradiction. Thus $(F_{\AA})_2$ and $(F_{\BB})_2$ are linearly independent,
and hence $F_{\AA}$ and $F_{\BB}$ are linearly independent.}
Therefore, $\mathcal V(Q)$ has dimension~$2$, and hence
$\mathcal P(Q)$ is a projective line. If $0\ne H=rF_{\AA}+sF_{\BB}$, then $H_2=r(F_{\AA})_2+s(F_{\BB})_2\ne0,$
so $H$ has degree~$2$. Since no nonzero element of $\mathcal V(Q)$ is
constant, {the space obtained by adjoining the constants is three-dimensional}.
Thus $\mathcal N(Q)$ is a projective plane.
\end{proof}

For a recent treatment of line pairs in pencils of real projective conics using geometric algebra, see \cite{Loucka}.
The next theorem extends the classical description of the pencil of conics through the vertices of a nondegenerate quadrangle to include the degenerate case also. 
 For $H\in\k[X,Y]$ and $w=(a,b)\in\k^2$, we write $D_wH=a\,\partial_XH+b\,\partial_YH$ for the directional derivative of $H$ in the direction $w$, and write $D_wH(p)$ for its value at a point $p$.

  \begin{theorem} \label{big pencil}

Let $Q$ be a quadrangle with vertices $v_1,v_2,v_3,v_4$ and sides $L_{ij}$. If $H\in\k[X,Y]$ has degree~$2$, then the conic $[H]$ is in $\mathcal P(Q)$ if and only if $H$ vanishes at the four vertices and, whenever $v_i=v_j$,  either $v_i$ is singular on $[H]$ or $L_{ij}$ is tangent to $[H]$ at $v_i$.

\end{theorem}

\begin{proof} 
If $Q$ is nondegenerate, the assertion follows from \cite[Theorem~4.25]{DragovicRadnovic} since
the pairs of opposite sides generate the pencil of conics through the four
vertices. Thus we may assume  for the rest of the proof that $Q$ is degenerate.
 We will use the fact that if $H(v_i)=0$ for some $i$, and $w$ is a direction vector for $L_{ij}$, then the condition that either $v_i$ is singular on $[H]$ or $L_{ij}$ is tangent to $[H]$ at $v_i$ is equivalent to $D_wH(v_i)=0$.  
 Let $W_Q$ be the vector space of polynomials in
$\k[X,Y]_{\leq2}$ that vanish at all four vertices and satisfy
$D_wG(v_i)=0$ whenever $v_i=v_j$ and $w$ is a direction vector for
$L_{ij}$. We will show  $W_Q = \mathcal V(Q)$. 
 
 Evaluating at the vertices and taking directional derivatives are linear in $H$, so for the forward implication, since the line-pair polynomials for the pairs of opposite sides of $Q$ span $\mathcal V(Q)$, it suffices to 
 show that each such line-pair polynomial vanishes at all four vertices and satisfies the given tangency property whenever two vertices coincide.  Let $\{S,S'\}$ be a pair of opposite sides, let $F_S$ and $F_{S'}$ be linear polynomials defining its two lines, and set $F=F_SF_{S'}$. Each vertex lies on at least one of $S$ and $S'$, so $F$ vanishes at all four vertices.
If also $v_i=v_j$, then whenever $L_{ij}$  is   $S$ or $S'$, we have $F|_{L_{ij}}=0$.  Since $F$ vanishes identically on $L_{ij}$, its directional derivative along $L_{ij}$ vanishes at $v_i$, so $D_wF(v_i)=0$. Otherwise, one of $S,S'$ passes through $v_i$ and the other passes through $v_j$. Hence both factors vanish at $v_i=v_j$.  By the product rule, $$D_wF(v_i)=D_wF_S(v_i)F_{S'}(v_i)+F_S(v_i)D_wF_{S'}(v_i)=0$$ for every direction $w$. Thus $v_i$ is singular on $[F]$.  This proves $\mathcal V(Q)\subseteq W_Q$, which verifies the forward implication.

 It remains to prove $W_Q\subseteq\mathcal V(Q)$. Since $Q$ is degenerate, Proposition~\ref{unpack} shows that it has a side of multiplicity one. After relabeling, write $S=L_{12}$ for this side, and let $F_{ij}$ be a linear polynomial defining $L_{ij}$. Since $S$ occurs only once, neither $v_3$ nor $v_4$ lies on $S$, and neither $L_{13}$ nor $L_{24}$ equals $S$. Set $F=F_{13}F_{24}$. Choose a nonzero direction vector $u$ for $S$. Every point of $S$ is uniquely of the form $v_1+tu$ for some $t\in\k$, so restriction to $S$ defines the
linear map $$ \rho:W_Q\longrightarrow\k[t]_{\leq2}:G\longmapsto G(v_1+tu). $$
For every $G\in W_Q$, if $v_1\ne v_2$, then $\rho(G)$ vanishes at the two distinct parameters corresponding to $v_1$ and $v_2$ because $G(v_1)=G(v_2)=0$. If $v_1=v_2$, then both vertices correspond to $t=0$, and the chain rule and the derivative condition along $S$ {imply} $\rho(G)(0)=G(v_1)=0$ and $\rho(G)'(0)=D_uG(v_1)=0.$
Thus $0$ is a double zero of $\rho(G)$. In either case, $\rho(G)$ has two
zeros counted with multiplicity. The polynomial $\rho(F)$ is a nonzero quadratic with these same zeros, so $\operatorname{im}\rho$ is contained in the one-dimensional space
$\operatorname{span}\{\rho(F)\}$.

 We claim that $\ker\rho=\operatorname{span}\{F_{12}F_{34}\}.$ The inclusion $\operatorname{span}\{F_{12}F_{34}\}\subseteq\ker\rho$ follows because $F_{12}F_{34}$ is an opposite-side polynomial and vanishes identically on $S$.  Conversely, if $G\in\ker\rho$, then $G$ vanishes identically on $S$. Since $F_{12}$ is a linear equation for $S$, it follows that $F_{12}$ divides $G$, so $G=F_{12}M$ for some polynomial $M$ of degree at most~$1$. Since $F_{12}(v_i)\ne0$ for $i=3,4$, we have $M(v_3)=M(v_4)=0$. If $v_3=v_4=p$ and $w$ is a direction vector for $L_{34}$, the derivative condition and the product rule {imply} $0=D_wG(p)=F_{12}(p)D_wM(p),$ because $M(p)=0$. Hence $M|_{L_{34}}$ has two zeros counted with multiplicity, whether or not $v_3=v_4$. Since this restriction has degree at most~$1$, it vanishes identically. Thus $M\in\operatorname{span}\{F_{34}\}$, so $G\in\operatorname{span}\{F_{12}F_{34}\}$. This proves the reverse inclusion and hence the claim that $ \ker\rho=\operatorname{span}\{F_{12}F_{34}\}. $    Therefore, $$ \dim W_Q=\dim\ker\rho+\dim\operatorname{im}\rho\leq2. $$ By Proposition~\ref{pencil independent}, $\dim\mathcal V(Q)=2$, while the forward implication {establishes} $\mathcal V(Q)\subseteq W_Q$. Consequently, $W_Q=\mathcal V(Q)$, which proves the theorem.
\end{proof}

\begin{corollary}\label{net vertex characterization}
Let $Q$ be a quadrangle with vertices $v_1,v_2,v_3,v_4$ and sides $L_{ij}$, and let $H\in\k[X,Y]$ have degree~$2$. Then $[H]$ is in the net $\mathcal N(Q)$ if and only if $$ H(v_1)=H(v_2)=H(v_3)=H(v_4) $$
and, whenever $v_i=v_j$ for some $i<j$, the line $L_{ij}$ is tangent at $v_i$ to the level conic $H=H(v_i)$ when that conic is nonsingular there.
\end{corollary}

\begin{proof}
 If $[H]\in\mathcal N(Q)$, write $H=F+c$ with $F\in\mathcal V(Q)$ and $c\in\k$. Theorem~\ref{big pencil} shows that $F$ vanishes at every vertex, so all four vertex values of $H$ equal $c$. The level conic $H=c$ is $[F]$, so the remaining tangency condition follows from the same theorem.
Conversely, set  $c=H(v_1)$ and $F=H-c$. Having equal vertex values implies that $F$ vanishes at every vertex, and the level conic $H=c$ is $[F]$. At every repeated vertex this conic is either singular or, by hypothesis, tangent to $L_{ij}$. Theorem~\ref{big pencil} therefore {implies} $F\in\mathcal V(Q)$, equivalently $[H]\in\mathcal N(Q)$.
\end{proof}

\section{Bisection of conics}\label{bisection section}

In this section we give a geometric characterization of the net of a quadrangle as the set of conics bisected by the sides of the quadrangle, thus reducing membership in the net of a quadrangle to a midpoint check involving only its six sides. We recall first the notion of bisectors as developed in \cite{OWQuad, OWPencil,OWProj}.

\smallskip

{\it Bisectors of conics.} Let $C=[H]$ be a conic and let $\ell$ be a line.  We say that $\ell$ \emph{crosses $C$} if it is not contained in $C$ and the equations defining $\ell$ and $C$ have a common solution in $\overline{\k}^2$, where $\overline{\k}$ is the algebraic closure of $\k$. When this occurs, the projective closure of $\ell$ meets the projective realization of $C$ in two points of $\P^2(\overline{\k})$, counted with multiplicity. If both points are finite, $\ell$ \emph{bisects $C$} at their midpoint. If one point is at infinity, $\ell$ bisects $C$ at the point at infinity on its projective closure. In either case, that point is called the \emph{bisection point} of $\ell$ for $C$. If $\ell$ does not cross $C$, it bisects $C$ at every point of the affine line $\ell$, and each such point is a bisection point.

Let $\mathcal F$ be a family of conics. If $\ell$ crosses at least one member of $\mathcal F$, we say that $\ell$ \emph{bisects $\mathcal F$ at $m$} when $m$ is its bisection point for every member of $\mathcal F$ that it crosses. If $\ell$ crosses no member of $\mathcal F$, it bisects $\mathcal F$ at every point of the affine line $\ell$. Any point at which $\ell$ bisects $\mathcal F$ is a \emph{bisection point} of $\ell$ for $\mathcal F$. We call $\ell$ a \emph{bisector of $\mathcal F$} if it bisects $\mathcal F$ at some bisection point.

\smallskip

In \cite{OWQuad,OWPencil,OWProj}, we called a bisection point of a bisector its ``midpoint.'' We call it bisection point here to avoid confusion with  the side midpoints of a quadrangle. A bisector that crosses a member of $\mathcal F$ has a unique bisection point; if it crosses no member, its bisection points are all the points on the line.

\smallskip

{\it Bisectors of quadrangles.}  A line $\ell$ crosses a line pair $\mathsf L=\{L,L'\}$ exactly when $\ell\notin\{L,L'\}$ and is not parallel to both lines. A line $\ell$ \emph{bisects a quadrangle $Q$ at $m$} if it bisects each pair of opposite sides of $Q$ at $m$. It is a \emph{bisector of $Q$} if it bisects $Q$ at some bisection point.

\smallskip

Each side of a quadrangle is a bisector of the quadrangle, and its side midpoint is a bisection point of this bisector. 
A quadrangle has many more bisectors, as will be evident from Theorem~\ref{degs} and subsequent results.

The following algebraic criterion for bisection of a conic will be crucial in what follows.   For a nondegenerate conic, the equivalence of (1) and (3) is the algebraic form of the classical theorem that the midpoints of a family of parallel chords of a conic lie on the conjugate diameter; see \cite[p.~88]{Coxeter}.

\begin{proposition}\label{reflection}
Let $C=[H]$ be a conic, let $\ell$ be a line, and let $w$ be a direction vector for $\ell$. The following are equivalent for a  point $m\in\ell$.
\begin{itemize}

\item[$(1)$] The line $\ell$ is a bisector of $C$ with bisection point $m$.

\item[$(2)$] $H(p)=H(q)$ for every  pair of points $p,q\in\ell$ having midpoint $m$.

\item[$(3)$] $D_wH(m)=0$.
\end{itemize}
  \end{proposition}

\begin{proof}
We can parameterize $\ell$ by $t\mapsto m+tw$, and since $H$ has degree at most~$2$,  we can write $H(m+tw)=at^2+bt+c,$ where $b=D_wH(m)$ and $a,c\in\k$. If the points $p,q\in\ell$ have midpoint $m$, then $p=m+tw$ and $q=m-tw$ for some $t\in\k$, and $$H(m+tw)-H(m-tw)=2bt.$$
 Since $\operatorname{char}(\k)\ne2$, condition~$(2)$ holds if and only if $b=0$, which is~$(3)$.

In homogeneous coordinates $[t:u]$ on the projective closure of $\ell$, the restriction of $H$ is
$at^2+btu+cu^2.$ If $a\ne0$, both intersection points of $\ell$ with $C$ are finite.  If their parameters are $r,r'\in\overline{\k}$, then $r+r'=-b/a$, so their midpoint is $m$ exactly when $b=0$.  If $a=0$ and $b\ne0$, then $H(m+tw)=bt+c$ has degree one in $t$, and its degree-two homogenization $u(bt+cu)$ has one finite zero and the point $[1:0]$ at infinity. Thus $\ell$ crosses $C$ and its bisection point is at infinity. If $a=b=0$, then $H|_\ell=c$ is constant. When $c=0$, the line $\ell$ is contained in $C$, and when $c\ne0$, it has no finite intersection with $C$. In either case $\ell$ does not cross $C$, so every point of $\ell$, including $m$, is a bisection point.  These three cases prove the equivalence of~$(1)$ and~$(3)$.
\end{proof}

\begin{corollary}
\label{conic butterfly property}
Let $C=[H]$ and $D=[K]$ be distinct conics in $\mathcal P(Q)$ for a quadrangle $Q$. Let $\ell$ be a line and $m \in \ell$. The following are equivalent.
\begin{itemize}
\item[$(1)$]    $\ell$ bisects both $C$ and $D$ at  $m$.
\item[$(2)$]   $\ell$ bisects $Q$ at $m$.
\item[$(3)$]  $\ell$ bisects every conic in $\mathcal N(Q)$ at $m$.
\end{itemize}
\end{corollary}

\begin{figure}[!ht]
\centering
\includegraphics[width=.7\textwidth]{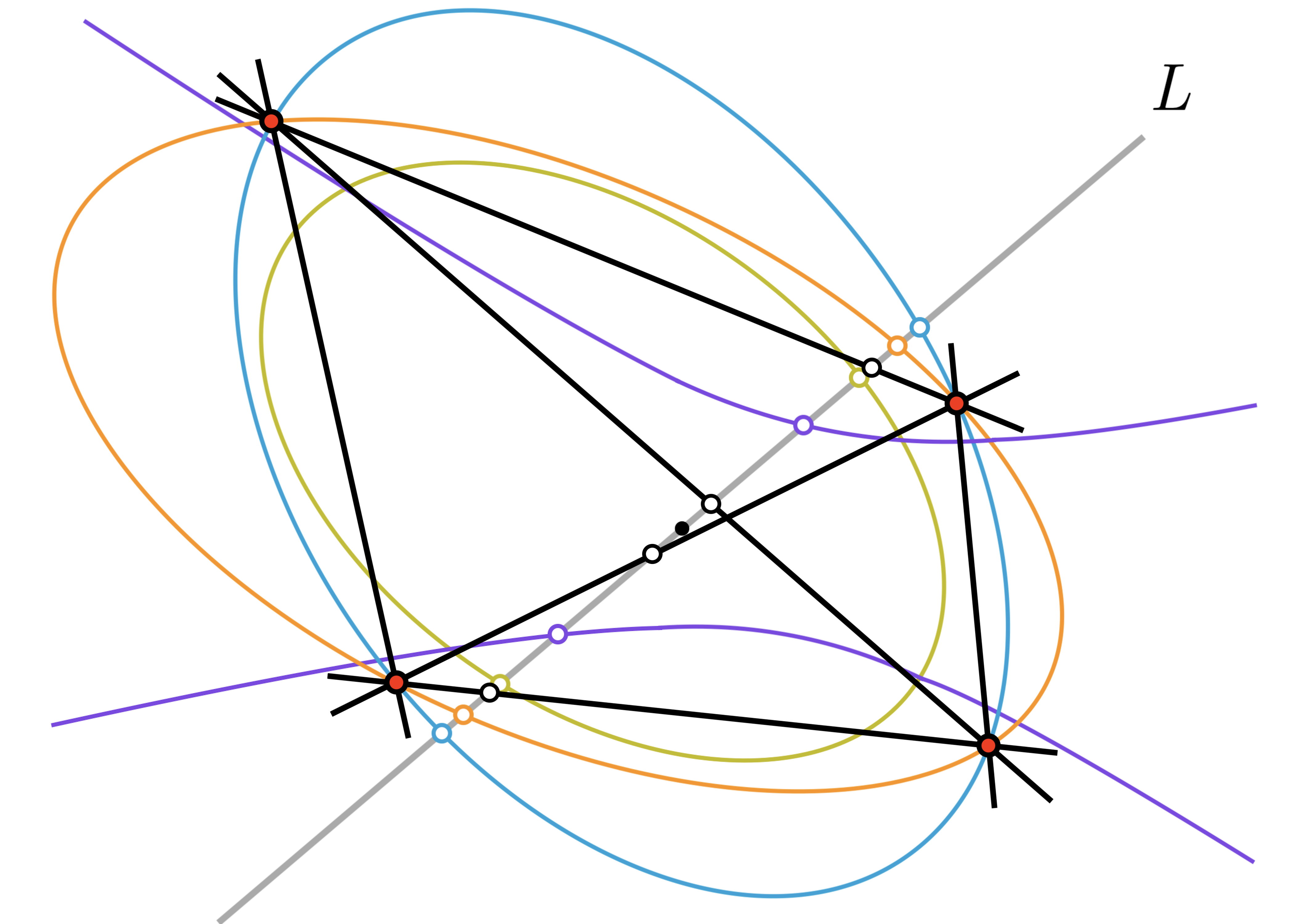}
\caption{A quadrangle and several conics from its net. Any line such as $L$ that bisects the quadrangle, or even just two conics in its net not from the same family of level curves,  must bisect all the conics in the net; see Corollary~\ref{conic butterfly property}.}
\label{conic butterfly figure}
\end{figure}

\begin{proof}
 Let $w$ be a direction vector for $\ell$. The linear functional $$\Phi_m:\mathcal V(Q)\oplus\k\mathbin{\cdot}1\longrightarrow\k:F\longmapsto D_wF(m)$$  vanishes on constants. By Proposition~\ref{reflection}, for each nonzero $F\in\mathcal V(Q)\oplus\k\mathbin{\cdot}1$, the line $\ell$ bisects the conic $[F]$ at $m$ exactly when $\Phi_m(F)=0$.

$(1)\Rightarrow(2)$: Since $C=[H]$ and $D=[K]$ are distinct members of $\mathcal P(Q)$, the polynomials $H$ and $K$ form a basis of $\mathcal V(Q)$.  By~$(1)$, $\Phi_m(H)=\Phi_m(K)=0$, so $\Phi_m$ vanishes on $\mathcal V(Q)$. In particular, it vanishes on a line-pair polynomial for every pair of opposite sides of $Q$. By Proposition~\ref{reflection}, $\ell$ bisects $Q$ at $m$.

$(2)\Rightarrow(3)$: Suppose that $\ell$ bisects $Q$ at $m$.  Line-pair polynomials for two distinct pairs of opposite sides form a basis of $\mathcal V(Q)$, and Proposition~\ref{reflection} shows that $\Phi_m$ vanishes on this basis. Hence it vanishes on $\mathcal V(Q)$ and, since it vanishes on constants, on $\mathcal V(Q)\oplus\k\mathbin{\cdot}1$. By Proposition~\ref{reflection}, $\ell$ bisects every member of $\mathcal N(Q)$ at $m$.

$(3)\Rightarrow(1)$: This follows since $C,D\in\mathcal P(Q)\subseteq\mathcal N(Q)$.
\end{proof}

\begin{remark}
 Corollary~\ref{conic butterfly property} contains a conic form of the classical Butterfly Theorem. If $C\in\mathcal P(Q)$ is nondegenerate and $m$ is the intersection of two nonparallel opposite sides of $Q$, then every line through $m$ bisects $C$ at $m$ if and only if it bisects $Q$ there; see \cite{CerinLines,CerinGianella}.
\end{remark}

\begin{corollary}\label{common tangency bisection}
Let $Q$ be a quadrangle. If a line $\ell$ is tangent at a point $m\in\k^2$ to two distinct conics in the net $\mathcal N(Q)$ that are nonsingular at $m$, then $\ell$ bisects every conic in $\mathcal N(Q)$ at $m$. Moreover, for each conic $[F]$ in the pencil $\mathcal P(Q)$, the level conic $F=F(m)$ is tangent to $\ell$ at $m$ whenever that level conic is nonsingular at $m$.
\end{corollary}

\begin{figure}[!htbp]
\centering
\includegraphics[width=\textwidth]{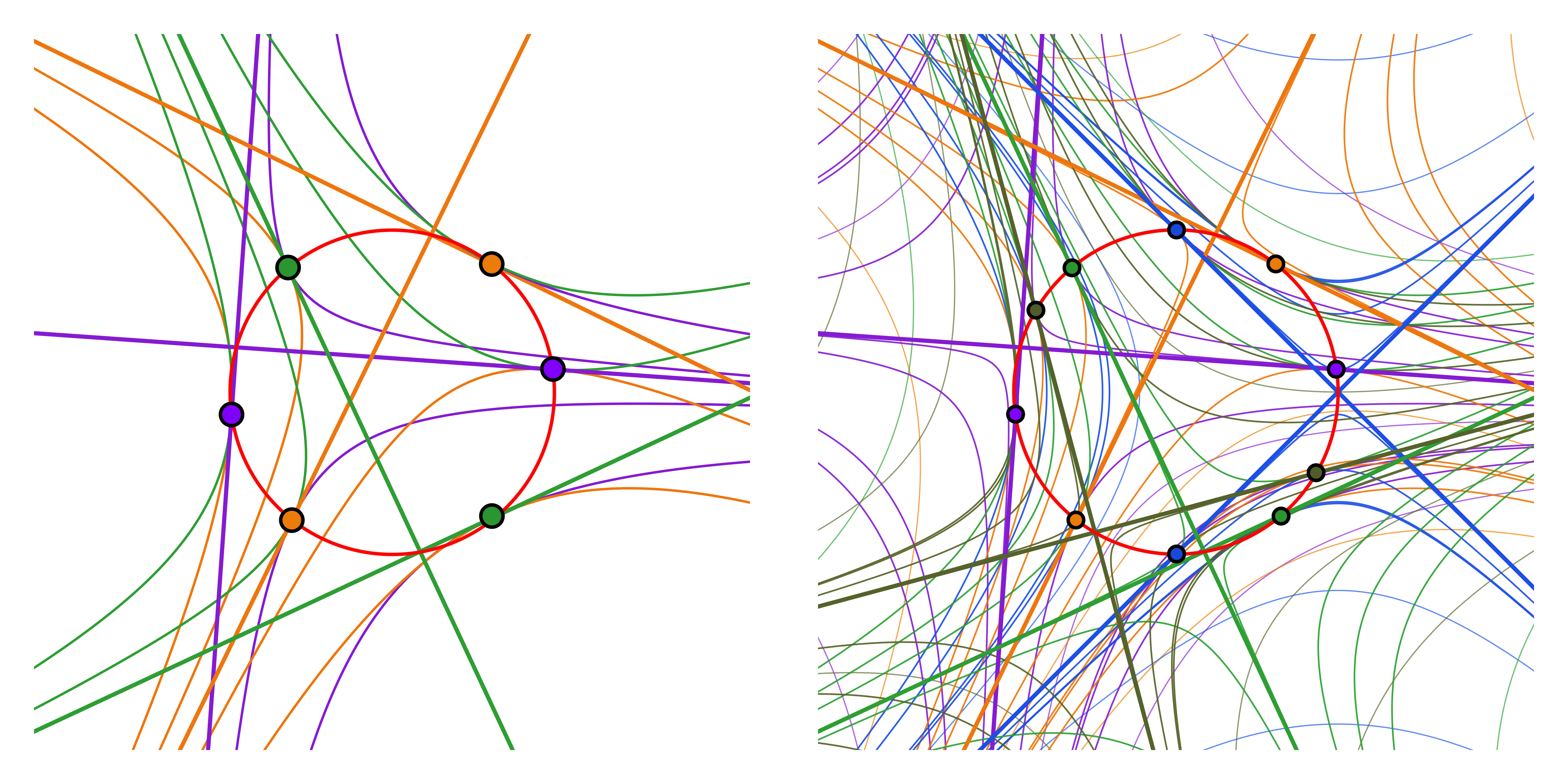}
\caption{Both figures consist of conics from the same quadrangle net, with fewer of the same conics shown at left. The red circle is the locus of bisection points (see Section 7). By  Corollary~\ref{common tangency bisection},  any line tangent at the same nonsingular point to two distinct conics in the net
 bisects every conic in the net at that point, as is the case for all the lines shown here.  Each such bisector  is tangent at its bisection point $m$ to each level conic $H=H(m)$ that is nonsingular there, where $[H]$ belongs to the net.
}
\label{halloween tangency figure}
\end{figure}

\begin{proof}
Write $C=[H]$ and $D=[K]$ for the two conics, and let $w$ be a direction vector for $\ell$. Since $C$ and $D$ are distinct, $H$ and $K$ are linearly independent. Moreover, $H(m)=K(m)=0$, so $1,H,K$ are linearly independent and hence form a basis of $\mathcal V(Q)\oplus\k\mathbin{\cdot}1$. Tangency implies $D_wH(m)=D_wK(m)=0$. Since $D_w1=0$, the linear functional $F\mapsto D_wF(m)$ vanishes on each element of the basis $1,H,K$ and hence on all of $\mathcal V(Q)\oplus\k\mathbin{\cdot}1$. By Proposition~\ref{reflection}, $\ell$ bisects every conic in $\mathcal N(Q)$ at $m$. For each $[F]\in\mathcal P(Q)$, the level conic $F=F(m)$ passes through $m$ and satisfies $D_w(F-F(m))(m)=0$, and so if the level conic is nonsingular at $m$, its tangent line at $m$ is $\ell$.
\end{proof}

%

%

The significance of the next theorem is that the conics in the net of a quadrangle are   those bisected by every side at its side midpoint. Thus the six sides, together with their side midpoints, {provide} a geometric test for membership in the net. In fact, it is even enough to test at four sides.

\begin{theorem}\label{conic net criterion}
Let $Q$ be a quadrangle and let $C=[H]$ be a conic. The following are equivalent.
\begin{enumerate}
\item[$(1)$] There are  two distinct pairs of opposite sides of $Q$ whose sides {bisect} $C$ at their side midpoints.
\item[$(2)$] Every side of $Q$ bisects $C$ at its side midpoint.
\item[$(3)$] The conic $C$ belongs to $\mathcal N(Q)$.
\end{enumerate}
Moreover, if $C\in\mathcal N(Q)$, then $C\in\mathcal P(Q)$ if and only if $H$ vanishes at a vertex of~$Q$.
\end{theorem}

\begin{proof}
Write the vertices, sides, and side midpoints of $Q$ as $v_i$, $L_{ij}$, and $m_{ij}$, respectively.

 $(1)\Rightarrow(3)$: Assume~$(1)$, and choose two such pairs of opposite sides $\AA$ and $\BB$. If $H$ is constant, then $C=[1]\in\mathcal N(Q)$, so we can  assume that $H$ is nonconstant. We first show that $H$ cannot have degree~$1$ and so must have degree~$2$.
Suppose that $H$ has degree~$1$. Proposition~\ref{reflection} shows that the nonzero linear functional $w\mapsto D_wH$ vanishes on the directions of all lines in $\AA$ or $\BB$. These include two nonparallel lines since if all these lines were parallel, the homogeneous quadratic parts of $F_{\AA}$ and $F_{\BB}$ would be proportional. Since $F_{\AA}$ and $F_{\BB}$ form a basis of the two-dimensional space $\mathcal V(Q)$, some nonzero element of $\mathcal V(Q)$ would then have degree at most~$1$, contrary to Proposition~\ref{pencil independent}. Direction vectors for two nonparallel lines form a basis of $\k^2$, so a linear functional vanishing on both must be zero. This contradicts the fact that $w\mapsto D_wH$ is nonzero when $H$ has degree~$1$. Hence $H$ has degree~$2$.

It remains to verify the vertex  and tangency conditions in Corollary~\ref{net vertex characterization}. For each side $L_{ij}$ in $\AA$ or $\BB$, Proposition~\ref{reflection} {implies} $H(v_i)=H(v_j)$, and it follows that the four vertex values have a common value $\lambda$.
Suppose $v_i=v_j$, and write $\{i,j,k,l\}=\{1,2,3,4\}$. Since no three vertices of a quadrangle coincide, $v_k$ and $v_l$ are distinct from $v_i$. Hence $L_{ik}=L_{jk}$ and $L_{il}=L_{jl}$, so the two pairs of opposite sides not containing $L_{ij}$ are equal. Since $\AA$ and $\BB$ are distinct, one of the chosen pairs contains $L_{ij}$. Proposition~\ref{reflection} therefore shows that the derivative of $H|_{L_{ij}}$ vanishes at $m_{ij}=v_i$. When the level conic $H=\lambda$ is nonsingular at $v_i$, this says  that $L_{ij}$ is tangent to it there. Corollary~\ref{net vertex characterization} now implies $C\in\mathcal N(Q)$, which proves $(1)\Rightarrow(3)$.

$(3)\Rightarrow(2)$: Suppose   $C\in\mathcal N(Q)$. Each side $L_{ij}$ bisects $Q$  at $m_{ij}$. Corollary~\ref{conic butterfly property} therefore shows that $L_{ij}$ bisects every member of $\mathcal N(Q)$ at $m_{ij}$, and in particular it bisects $C$ there.

$(2)\Rightarrow(1)$: This is clear.

It remains to prove the final assertion.  Suppose  $C\in\mathcal N(Q)$, and write $H=K+\lambda$ with $K\in\mathcal V(Q)$ and $\lambda\in\k$. Then $H(v_i)=\lambda$ for every $i$.  Proposition~\ref{pencil independent} shows that no nonzero constant belongs to $\mathcal V(Q)$, so  $C=[H]\in\mathcal P(Q)$ exactly when $\lambda=0$, equivalently, when $H$ vanishes at one, and hence every, vertex  of $Q$.
\end{proof}

\begin{definition} 
Two quadrangles are \emph{entangled} if they have the same net.
\end{definition}

Using Theorem~\ref{conic net criterion}, we can {state} our first geometric characterization of entanglement. See Figure~\ref{entangled quads} for an illustration of the corollary.

\begin{corollary}\label{entanglement bisection}
 Two quadrangles $Q_1$ and $Q_2$ are entangled if and only if every side of $Q_1$ bisects $Q_2$ at its side midpoint. In particular, this condition is symmetric in $Q_1$ and $Q_2$.
\end{corollary}

\begin{proof}
Write  $\mathcal U_i=\mathcal V(Q_i)\oplus\k\mathbin{\cdot}1$ and $\mathcal N_i=\mathcal N(Q_i)$ for $i=1,2$. Suppose first that every side of $Q_1$ bisects $Q_2$ at its side midpoint. Choose  two distinct pairs of opposite sides $\AA$ and $\BB$ of $Q_2$, and let $F_{\AA}$ and $F_{\BB}$ be line-pair polynomials defining them. Since every side of $Q_1$ bisects each of the conics $[F_{\AA}]$ and $[F_{\BB}]$ at its side midpoint, Theorem~\ref{conic net criterion} {implies} $F_{\AA},F_{\BB}\in\mathcal U_1$. By Proposition~\ref{pencil independent},
$$\mathcal U_2=\operatorname{span}\{1,F_{\AA},F_{\BB}\} \subseteq\mathcal U_1.$$
Both spaces have dimension~$3$, so they are equal and the quadrangles are entangled.

Conversely, suppose the quadrangles are entangled, so that $\mathcal N_1=\mathcal N_2$. Each pair of opposite sides of $Q_2$ defines a conic in $\mathcal N_2$, and hence in $\mathcal N_1$. Theorem~\ref{conic net criterion} shows that every side of $Q_1$ bisects each of these pairs at its side midpoint. Thus every side of $Q_1$ bisects $Q_2$ at its side midpoint.
\end{proof}

\section{The bisector field of a quadrangle}\label{bisector field section}

In this section, we show that the line pairs in the net $\mathcal N(Q)$ form a bisector field $\mathcal B(Q)$ whose component lines are the bisectors of $Q$. To establish this connection, we relate quadrangles without a triple side to complete quadrilaterals having the same bisectors and bisector field, and treat triple-side quadrangles separately.

\begin{definition}
A \emph{bisector field} is a family $\mathcal F$ of line pairs satisfying three conditions. Every component line of a pair in $\mathcal F$ bisects {the family} $\mathcal F$. The family $\mathcal F$ is not one in which every pair is a translation of every other pair, all its component lines pass through one common finite point, or all its component lines are parallel. Finally, $\mathcal F$ is maximal among families having the first property.
\end{definition}

See Figures~\ref{cardioid}, \ref{triangle field image} and \ref{trapezoid}  for an example of a bisector field. The following proposition {identifies ${\mathcal B}(Q)$ as a bisector field}.

\begin{proposition} \label{pencil} For a quadrangle $Q$,
$\mathcal B(Q)$ is a bisector field.
\end{proposition} 

\begin{proof}
A line pair $\mathsf L$ is a {degeneration} of a conic $C$ if $C=[F_{\mathsf L}+c]$ for some line-pair polynomial $F_{\mathsf L}$ defining $\mathsf L$ and some $c\in\k$. If the lines in $\mathsf L$ are nonparallel, then $\mathsf L$ is the asymptote pair of the hyperbola $C$ and is its unique degeneration. If they are parallel, the degenerations are  the pairs of parallel lines having the same midline\footnote{If two lines in a line pair are parallel and have equations $f=a$ and $f=b$, their \emph{midline} is the line $f=(a+b)/2$.} as $\mathsf L$, including the midline paired with itself \cite[Proposition~2.2]{OWPencil}. By \cite[Definitions~2.1 and~3.1 and Proposition~3.2]{OWPencil},  the set of degenerations of the members in $\mathcal P(Q)$  (called the asymptotic pencil in \cite{OWPencil}) is   the set of line pairs in $\mathcal N(Q)$. This set contains the three pairs of opposite sides of $Q$. Since $Q$ is a quadrangle, the lines in the pairs include three nonconcurrent, pairwise nonparallel lines, which excludes each of the three elementary families in the definition. The proposition now follows from \cite[Theorem~6.3]{OWPencil}.
\end{proof}

Accordingly, we call $\mathcal B(Q)$ the \emph{bisector field of $Q$}.
In the next section, Corollary~\ref{bisector field} will show that   $\mathcal B(Q)$ is the family of bisector pairs of $Q$. For this purpose, we use a bridge to the complete quadrilaterals studied in \cite{OWQuad}, where a \emph{complete quadrilateral} consists of four distinct lines arranged into two pairs of opposite sides such that not all four lines are concurrent and every pair of adjacent sides meets. Its four vertices and two diagonals are then determined, and adjoining the diagonals produces a quadrangle.  If the four vertices of the quadrilateral are distinct, the proof of \cite[Proposition~3.4]{OWQuad} shows that the quadrilateral and the completed quadrangle have the same bisectors, with the same {bisection points}. If two vertices of the quadrilateral coincide, the same conclusion follows directly since its diagonal pair repeats one pair of opposite sides.

Part of the reason for bridging to the quadrilateral case is to make use of the following lemma, which is collected from some of the main results in \cite{OWQuad,OWPencil}. 

 \begin{lemma}\label{quadrilateral bisector pairing}
Let $Q_0$ be a complete quadrilateral. The lines occurring in its bisector field are  its bisectors, and every bisector belongs to a unique pair in the field.
\end{lemma}

\begin{proof}
This follows from \cite[Lemma~6.1, Theorem~6.4, and Corollary~6.7]{OWQuad} and \cite[Theorem~6.3]{OWPencil}.
\end{proof}

 Theorem~\ref{weird} shows  a quadrangle that is not a triple-side quadrangle  can be traded for a complete quadrilateral for purposes of working with a net. The triple-side case behaves differently and must be handled directly. The following lemma {explicitly describes} both its bisectors and its bisector field; see Figure~\ref{triangle field image}.

\begin{lemma} \label{triangle bf}
Let $Q$ be a triple-side quadrangle, let $v$ be the vertex of $Q$ that is not on the triple side $T$  of $Q$, and let $R$ be the line through $v$ parallel to $T$.
\begin{itemize}
\item[$(1)$]   The bisectors of $Q$ are  the lines through $v$ and the lines parallel to  $T$.  Every point of $T$ is a {bisection point} of $T$. If $\ell\parallel T$ and $\ell\ne T$, then the point at infinity on $\ell$ is its unique {bisection point}. If $\ell$ passes through $v$ and is not parallel to $T$, then its unique {bisection point} is the midpoint of $v$ and $\ell\cap T$.
\item[$(2)$] The nonparallel members of $\mathcal B(Q)$ are   the pairs $\{T,M\}$, where $M$ ranges over the lines through $v$ that are not parallel to $T$. Its remaining members are   the pairs of lines parallel to $T$ having the same midline as $\{T,R\}$; these include $\{T,R\}$ and the double midline.
\end{itemize}
\end{lemma}

\begin{figure}[ht] 
     \begin{center}
\includegraphics[width=.95\linewidth]{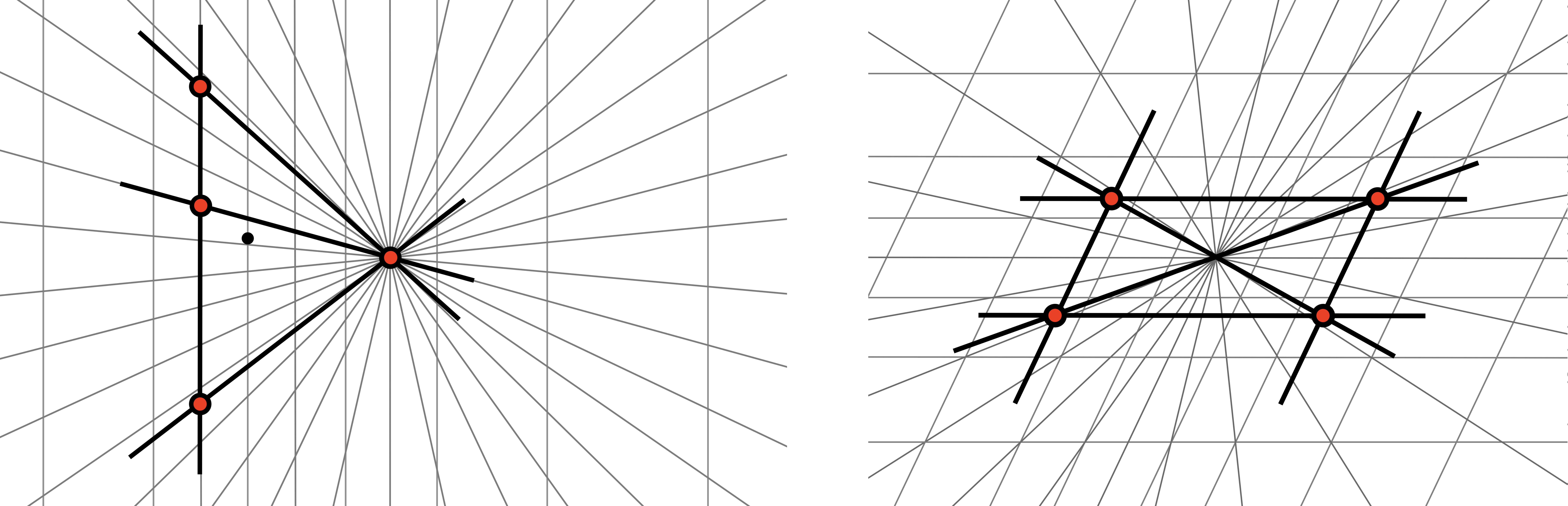}
 \end{center}
\caption{Bisectors (in gray) of two special quadrangles over $\mathbb R$. On the left is a triple-side quadrangle. The red points are vertices and the small point is the centroid. On the right is a parallelogram. }
 \label{triangle field image}
 \label{parallelogram figure}
\end{figure}

\begin{proof}
{
Let $L_1,L_2,L_3$ be the non-triple sides of $Q$. Each passes through $v$, and the pairs of opposite sides are $\{T,L_i\}$, $1\leq i\leq3$. At least two of the lines $L_i$ are distinct, since otherwise  $Q$ would have only two distinct sides and would not be a quadrangle. The line $T$ crosses no pair of opposite sides, so  every point of $T$ is a {bisection point} of $T$. If $\ell\parallel T$ and $\ell\ne T$, then $\ell$ meets $T$ at the point at infinity on $\ell$ and  meets every $L_i$. Hence the point at infinity on $\ell$ is its unique {bisection point}. If $\ell$ passes through $v$ and is not parallel to $T$, then at least one $L_i$ is distinct from $\ell$. Every pair of opposite sides crossed by $\ell$ meets it at $v$ and $\ell\cap T$. Thus $\ell$ has the midpoint of these two points as its unique {bisection point}.

Conversely, suppose that a bisector $\ell$ is neither parallel to $T$ nor passes through~$v$, and put $p=\ell\cap T$. If $\ell$ is parallel to some $L_i$, choose $L_j\ne L_i$. The pair $\{T,L_i\}$ {determines} an infinite {bisection point}, whereas $\{T,L_j\}$ {determines} a finite {bisection point}, a contradiction. Hence $\ell$ meets every $L_i$. If $m$ is the common {bisection point} of $\ell$, then every $\ell\cap L_i$ is the reflection of $p$ about $m$, so these intersection points coincide. Two distinct lines among the $L_i$ meet only at $v$, and hence $v\in\ell$, again a contradiction.

It remains to prove~(2). The line-pair members of $\mathcal P(Q)$ are the pairs $\{T,M\}$ with $M$ through $v$. If $M$ is not parallel to $T$, the description of degenerations in the proof of Proposition~\ref{pencil} shows that $\{T,M\}$ contributes only itself to the bisector field. These are  the nonparallel members in~(2). The only parallel member of $\mathcal P(Q)$ is $\{T,R\}$, and it contributes precisely the pairs of lines parallel to $T$ having the same midline as $\{T,R\}$. This family includes $\{T,R\}$ itself and the double midline. This proves~(2).
}
\end{proof}

The bisector field of a parallelogram $Q$ also has a particularly simple description, illustrated in Figure~\ref{parallelogram figure}. For each pair of parallel opposite sides, it contains every pair of parallel lines having the same midline, as well as the midline paired with itself. Its remaining pairs are the pairs of diagonals of the parallelograms whose pairs of opposite sides arise in this way. Consequently, the bisectors are  the lines through the centroid together with the lines parallel to the sides of the parallelogram. See also \cite[Example~6.3]{OWQuad}.

Describing the bisectors of quadrangles gets more complicated for quadrangles that are not parallelograms and have no triple side. For example, over the field of real numbers, the bisectors can appear as lines tangent to a parabola or to a rational cuspidal curve, as in Figure~\ref{deltoid}; see \cite{OWProj}.
 \begin{figure}[ht] 
     \begin{center}
\includegraphics[width=\textwidth]{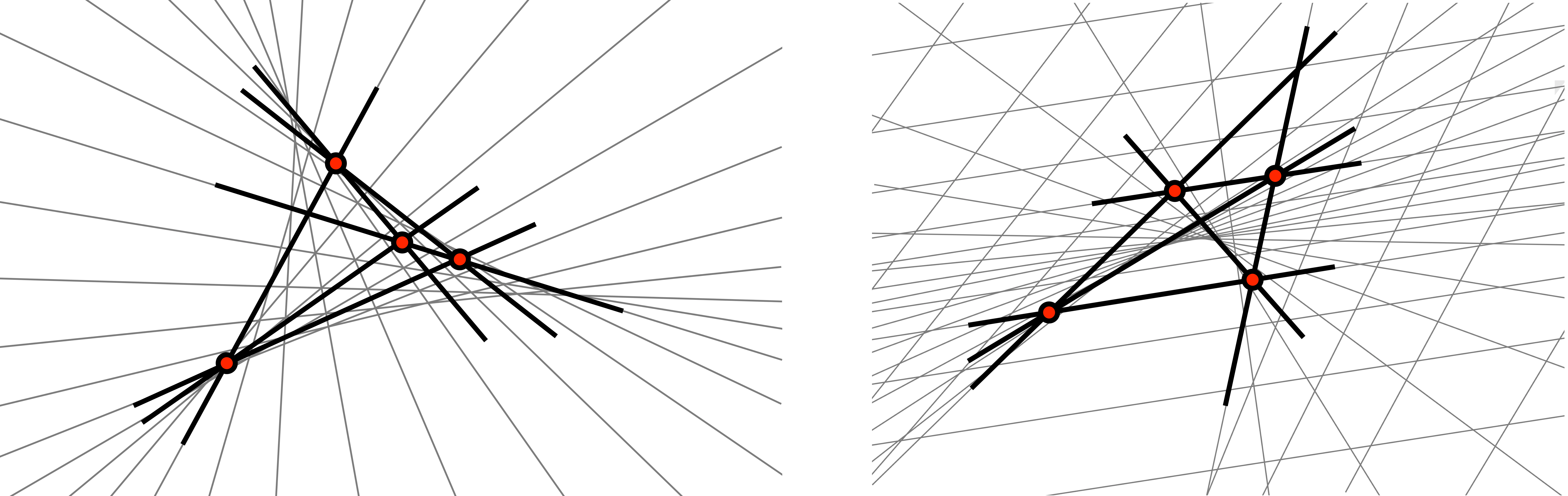}
 \end{center}
\caption{Bisectors of two real quadrangles. On the left, the bisectors envelop the Steiner deltoid. On the right, the bisector field of a {nonparallelogram trapezoid} consists of tangents to a parabola together with a pencil of parallel lines.}
 \label{deltoid} \label{trapezoid}
\end{figure}

  \begin{theorem} \label{weird} The following are equivalent for a quadrangle $Q$.
\begin{itemize}
\item[$(1)$] $Q$ has no triple side.
\item[$(2)$] Each bisector of $Q$ has a unique finite {bisection point}.
\item[$(3)$] {There is a complete quadrilateral $Q_0$ having the same bisectors and bisector field as $Q$, with the same bisection points for every common bisector.}
\end{itemize}
 
\end{theorem}

\begin{proof}
{
$(1)\Rightarrow(2)$: Let $\ell$ be a bisector of a quadrangle $Q$ with no triple side. {We claim that some pair of opposite sides of $Q$ consists of two sides neither of which is parallel to $\ell$.} Otherwise, from each of the three pairs of opposite sides choose a side parallel to $\ell$, possibly equal to $\ell$. The three chosen sides lie in distinct pairs of opposite sides and hence are pairwise adjacent. Since adjacent sides meet and all three are parallel to $\ell$, they coincide. The quadrangle $Q$ would then have a triple side, a contradiction. Therefore, $Q$ has a pair $\{L,L'\}$ of opposite sides such that neither $L$ nor $L'$ {is parallel to $\ell$.}  Both intersections $\ell\cap L$ and $\ell\cap L'$ are finite, so $\ell$ crosses this pair. Since $\ell$ bisects $Q$, the midpoint of these two points is its unique {bisection point}. Thus~$(2)$ holds.

$(2)\Rightarrow(1)$: If $Q$ is a triple-side quadrangle with triple side $T$, Lemma~\ref{triangle bf} shows that a line parallel to $T$ and distinct from $T$ has an infinite {bisection point}, while every point of $T$ is a {bisection point} of $T$. Thus if~(1) fails, so does $(2)$.

 $(1)\Rightarrow(3)$: Suppose that $Q$ has no triple side. We construct a quadrilateral $Q_0$ and its two pairs of opposite sides $\AA$ and $\BB$ in each of the three possible cases.
If $Q$ is nondegenerate, let $Q_0$ be a quadrilateral whose completed quadrangle is $Q$, and let $\AA$ and $\BB$ be its two pairs of opposite sides.

Suppose instead that $Q$ is degenerate. By Proposition~\ref{unpack}, $Q$ satisfies case~(a) or case~(c) of that proposition. In case~(a), two of the pairs of opposite sides of $Q$ are the same pair $\AA$, while the third is a distinct pair $\BB$. The two lines in $\AA$ and the two lines in $\BB$ are distinct. Every line in $\AA$ is adjacent to every line in $\BB$, so each such pair meets. Three of the four lines pass through the double vertex, while the fourth does not, so the four lines are not concurrent. Thus $\AA$ and $\BB$ are the pairs of opposite sides of a quadrilateral $Q_0$.
Its diagonal pair is $\AA$, {so $Q$ is obtained by adjoining the diagonals of $Q_0$}.

In case~(c), let $\ell$ be the quadruple side and let $\AA$ be the remaining pair of opposite sides. Choose a pair $\BB=\{B,B'\}$ of distinct parallel lines parallel to $\ell$ and having $\ell$ as their midline. The two lines in $\AA$ are distinct and meet $\ell$, whereas the lines in $\BB$ are distinct and parallel to $\ell$. Hence these four lines are distinct, every line in $\AA$ meets every line in $\BB$, and the four lines are not concurrent. Thus $\AA$ and $\BB$ are the pairs of opposite sides of a quadrilateral $Q_0$. If $F_\ell=0$ defines $\ell$, then a line-pair polynomial for $\BB$ has the form $F_\ell^2-c$ for some $c\in\k^\times$.

  In the nondegenerate, double-side, and quadruple-side cases, line-pair polynomials $F_{\AA}$ and $F_{\BB}$ for the pairs of opposite sides of $Q_0$ satisfy
 \begin{equation}\label{auxiliary net span}
\mathcal V(Q)\oplus\k\mathbin{\cdot}1
=\operatorname{span}\{F_{\AA},F_{\BB},1\}.
\end{equation}
Outside case~(c), this follows from Proposition~\ref{pencil independent}. In case~(c), it follows since $F_{\BB}=F_\ell^2-c$ and
$$\operatorname{span}\{F_{\AA},F_\ell^2-c,1\}=\operatorname{span}\{F_{\AA},F_\ell^2,1\}.$$
By \cite[Theorem~6.4]{OWQuad} and \cite[Theorem~6.3]{OWPencil}, the bisector field of $Q_0$ is therefore $\mathcal B(Q)$. Let $m$ be a  point on a line $L$, and let $w$ be a direction vector for $L$. Proposition~\ref{reflection} shows that $L$ bisects $Q$ at $m$ exactly when {the functional $H\mapsto D_wH(m)$} vanishes on $\mathcal V(Q)$, while it bisects $Q_0$ there exactly when $D_wF_{\AA}(m)=D_wF_{\BB}(m)=0$. Since $D_w(1)=0$,  equality~$(\ref{auxiliary net span})$ makes these conditions equivalent. Every bisector of $Q_0$ has a unique finite {bisection point} by \cite[Proposition~3.2(1)]{OWQuad}, and the implication $(1)\Rightarrow(2)$ {establishes} the same conclusion for $Q$. Hence $Q$ and $Q_0$ have the same bisectors, with the same {bisection points}, {as required in~$(3)$}.

\smallskip

{$(3)\Rightarrow(2)$: By \cite[Proposition~3.2(1)]{OWQuad}, every bisector of $Q_0$ has a unique finite bisection point. Since $Q$ and $Q_0$ have the same bisectors and bisection points, the same holds for $Q$.}
}
\end{proof}

\begin{corollary}\label{component bisectors}
The lines in the pairs in $\mathcal B(Q)$ are   the bisectors of $Q$.
\end{corollary}

\begin{proof}
If $Q$ has no triple side, then {Theorem~\ref{weird} shows there is a complete quadrilateral $Q_0$ having the same bisectors and bisector field as $Q$.}  {The assertion for $Q$ follows from Lemma~\ref{quadrilateral bisector pairing}}. If instead $Q$ has a triple side, Lemma~\ref{triangle bf}(2) shows that the component lines of $\mathcal B(Q)$ are exactly the lines through the vertex off the triple side and the lines parallel to the triple side. By Lemma~\ref{triangle bf}(1), these are  the bisectors of $Q$.
\end{proof}
 
{{In the next section, we identify directly which bisectors occur together in {the} line pairs {of $\mathcal B(Q)$}.} Also, while Theorem~\ref{weird} suggests that it would simplify matters to exclude the exceptional triple-side case from the definition of a quadrangle, it is an important boundary case when working with moduli spaces of quadrangles, as in \cite{OWManifold} and related work.}

\section{Bisector pairs of quadrangles}\label{bisector pairs section}

 In this section, we 
 formalize the idea of a bisector pair  of a quadrangle $Q$, and show that the bisector field $\mathcal B(Q)$ of $Q$ is the set of its bisector pairs, thus giving a simpler reinterpretation of the bisector field in terms of line pairs bisected by the sides of $Q$.

\begin{definition} \label{bp def} 
A line pair is a {\it bisector pair} of a quadrangle $Q$ if the lines in the pair are bisectors of $Q$ and every side of $Q$ bisects the pair at its side midpoint.
\end{definition}
  \begin{figure}[ht] 
     \begin{center}
\includegraphics[width=.75\textwidth]{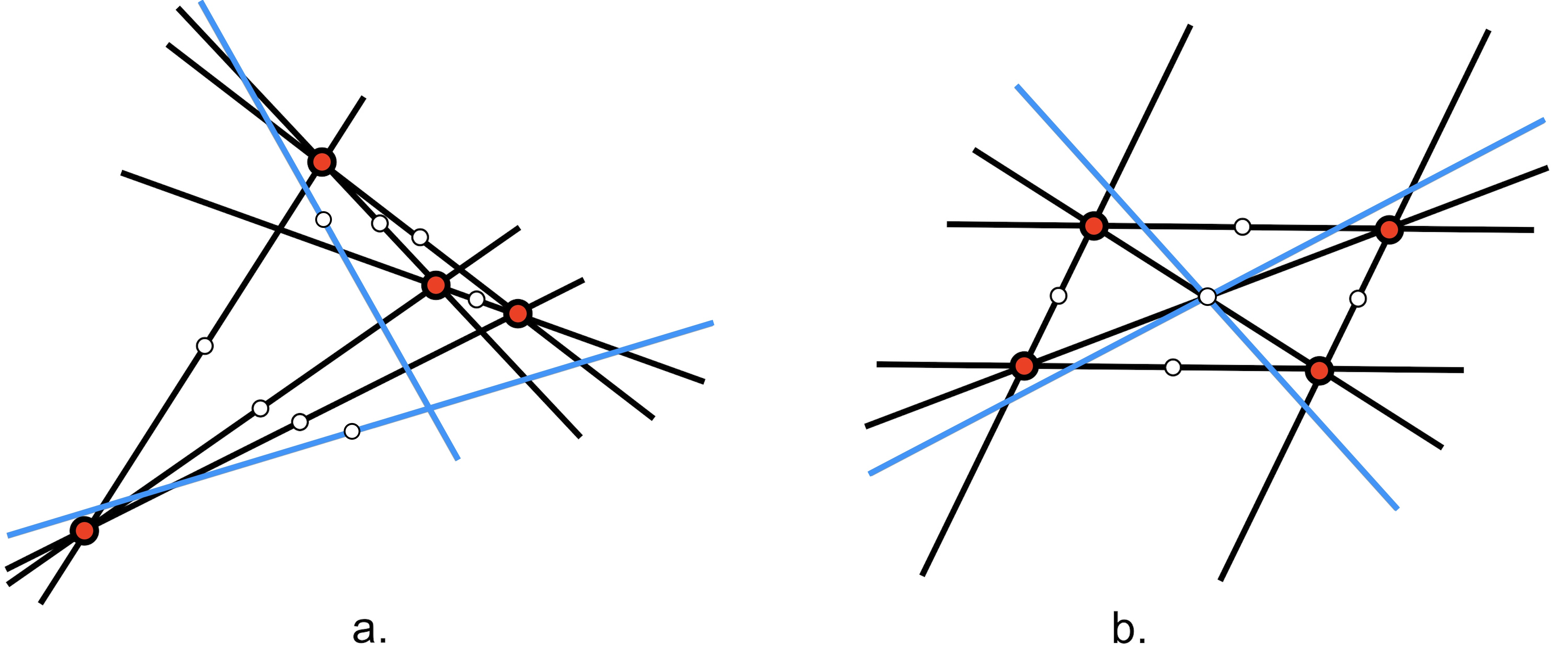}
 \end{center}
\caption{Bisector pairs of two quadrangles. Bisection points are white, and the bisector pairs blue. Every side bisects the blue pair at its side midpoint. By Theorem~\ref{tricky}, this implies that the blue lines must in turn bisect the quadrangle.}
 \label{bisector pair figure}
\end{figure}
Figure~\ref{bisector pair figure} illustrates the definition. The pairs of opposite sides of a quadrangle are trivially bisector pairs. By definition, a bisector pair must be bisected by every side of $Q$ at its side midpoint. Thus, if two sides $L_{ij}$ and $L_{i'j'}$ of a quadrangle are the same line, that line must bisect the pair at the midpoint of $v_i$ and $v_j$ and also at the midpoint of $v_{i'}$ and $v_{j'}$. For example, if a line occurs three times as a triple side, then it must bisect the line pair at each of the three corresponding side midpoints.
 
The next theorem {describes} a bisector pair {in three equivalent ways}. Its main import is that bisection of a pair of lines by the sides of $Q$ at their side midpoints forces the   lines in the pair to be bisectors of $Q$. In short, bisected pairs are bisector pairs.

\begin{theorem}\label{line-pair net members}\label{degs}\label{tricky}
The following are equivalent for a quadrangle $Q$ and a line pair~${\mathsf L}$.
\begin{enumerate}
\item[$(1)$] ${\mathsf L}$ is a bisector pair of $Q$.
\item[$(2)$] ${\mathsf L}$ is bisected by every side of $Q$  at its side midpoint.
\item[$(3)$] ${\mathsf L}\in\mathcal B(Q)$; equivalently, ${\mathsf L}$ is a line-pair member of $\mathcal N(Q)$.
\end{enumerate}
\end{theorem}

\begin{proof}
Let $F$ be a line-pair polynomial defining ${\mathsf L}$. Condition~(1) implies~(2) by definition, and Theorem~\ref{conic net criterion} shows~(2) holds if and only if   $[F]\in\mathcal N(Q)$, which is equivalent to ${\mathsf L}\in\mathcal B(Q)$, so it is condition~(3). To see that (3) implies (1), suppose that ${\mathsf L}\in\mathcal B(Q)$. Then ${\mathsf L}$ is a line-pair member of $\mathcal N(Q)$, so Theorem~\ref{conic net criterion} {implies}~(2). Corollary~\ref{component bisectors} also shows that both   lines in ${\mathsf L}$ are bisectors of $Q$. Thus ${\mathsf L}$ is a bisector pair of $Q$.
\end{proof}

{The quick  proof of Theorem~\ref{tricky} hides the difficulty of the implication $(2)\Rightarrow(1)$. The hypothesis of (2) says that every side of $Q$ bisects ${\mathsf L}$ at its side midpoint, whereas the conclusion in (1) says that the   lines in ${\mathsf L}$ bisect $Q$. This symmetry ultimately rests on the description in Section~5 of the line pairs in the net as a bisector field and on Corollary~\ref{component bisectors}. These facts use substantial results from \cite{OWQuad,OWPencil}, and we do not know a simple direct proof.}
 In any case,
Theorem~\ref{tricky} allows us to  reinterpret the bisector field of a quadrangle without reference to its pencil or net of conics. 


\begin{corollary} \label{bisector field} \label{reflexive} For a quadrangle $Q$,   $\mathcal B(Q)$ is the family of bisector pairs of~$Q$. In particular, every bisector of $Q$ bisects every bisector pair of $Q$.   
\end{corollary}

\begin{proof}
By Definition~\ref{bf def}, $\mathcal B(Q)$ consists  of the line pairs in ${\mathcal N}(Q)$, and Theorem~\ref{tricky} shows that the same is true of the family of bisector pairs of $Q$. By Corollary~\ref{component bisectors}, every bisector of $Q$ is a line of one of these pairs, so the last assertion now follows from the definition of a bisector field.
\end{proof}

By definition, two quadrangles are entangled if they have the same net. In fact, it is enough that the quadrangles have the same bisector field:

\begin{theorem}\label{nice}
For a quadrangle $Q$, the net $\mathcal N(Q)$ is determined by $\mathcal B(Q)$. Consequently, two quadrangles $Q_1$ and $Q_2$ are entangled if and only if $\mathcal B(Q_1)=\mathcal B(Q_2)$.
\end{theorem}

\begin{proof} Since entangled quadrangles $Q_1$ and $Q_2$ have the
same net, $\mathcal B(Q_1)=\mathcal B(Q_2)$. Conversely, suppose $\mathcal B(Q_1)=\mathcal B(Q_2)$.  For a quadrangle $Q$, choose two distinct pairs of opposite sides $\AA$ and $\BB$, and let $F_{\AA}$ and $F_{\BB}$ be line-pair polynomials defining them.  {By Definition~\ref{bf def} and Proposition~\ref{pencil independent},}
$$\mathcal N(Q)=\mathbb P\bigl(\operatorname{span}\{1,F_{\AA},F_{\BB}\}\bigr).$$
 The pairs $\AA$ and $\BB$ belong to $\mathcal B(Q)$, while every member of $\mathcal B(Q)$ belongs to $\mathcal N(Q)$. Hence $\mathcal N(Q)$ is   the projective span of $[1]$ and the points represented by the members of $\mathcal B(Q)$. {By this description,} $\mathcal N(Q_1)=\mathcal N(Q_2)$. Thus $\mathcal B(Q)$ determines $\mathcal N(Q)$, and quadrangles with the same bisector field are entangled.
\end{proof}

{The next theorem shows that, except for triple-side quadrangles, every bisector has a unique partner in a bisector pair.} We also characterize trapezoids and parallelograms by the existence, respectively, of a parallel bisector pair and of two distinct bisectors with a common finite {bisection point}.

\begin{theorem} \label{pairing structure} \label{special pair characterizations}
\label{mid centroid} \label{unique pair} \label{old}
Let $Q$ be a quadrangle.
\begin{enumerate}
\item[$(1)$]  Each bisector of $Q$ belongs to a unique bisector pair of $Q$ if and only if $Q$ has no triple side. In this case,  the centroid of $Q$ is the midpoint of the {bisection points} in every bisector pair.

\item[$(2)$] \label{parallel pair} {The quadrangle $Q$ is a trapezoid if and only if $Q$ has no triple side and $Q$ has a bisector pair whose lines are parallel. In this case, the lines in each parallel bisector pair are parallel to a pair of parallel opposite sides of $Q$, and every line parallel to a pair of parallel opposite sides of $Q$ is a bisector of $Q$.}

\item[$(3)$] \label{parallelogram pairing} The quadrangle $Q$ is a parallelogram if and only if it has two distinct bisectors with a common finite {bisection point}. This common {bisection point} is the centroid of $Q$.
\end{enumerate}
\end{theorem}

\begin{proof} (1) Suppose that $Q$ has no triple side, and let $Q_0$ be the {complete quadrilateral} constructed in the proof of Theorem~\ref{weird}. Then $Q$ and $Q_0$ have the same bisectors and bisector field, with the same {bisection points}.  By Corollary~\ref{bisector field} and Lemma~\ref{quadrilateral bisector pairing}, the bisector pairs of $Q$ are the pairs in the bisector field of $Q_0$. The lemma also shows that every bisector of $Q_0$, and hence every bisector of $Q$, belongs to a unique such pair.

 It remains to prove the assertion about the centroid. We first show that $Q$ and $Q_0$ have the same centroid. Except in case~(c) of Proposition~\ref{unpack}, they have the same four vertices, counted with multiplicity, so this is immediate. In case~(c), write $\AA=\{A,A'\}$ for the remaining pair of opposite sides of $Q$. The pair $\BB=\{B,B'\}$ consists of parallel lines having the quadruple side $\ell$ as their midline. Hence $A\cap\ell$ is the midpoint of the two vertices $A\cap B$ and $A\cap B'$ of $Q_0$, and similarly $A'\cap\ell$ is the midpoint of $A'\cap B$ and $A'\cap B'$. Thus the two double vertices $A\cap\ell$ and $A'\cap\ell$ of $Q$ are the midpoints of the opposite sides $A$ and $A'$ of $Q_0$. The centroid of $Q_0$ is the midpoint of these two side midpoints, while the centroid of $Q$ is the midpoint of its two double vertices. Therefore $Q$ and $Q_0$ have the same centroid in this case as well. Finally, by \cite[Definition~6.2]{OWQuad}, the centroid of $Q_0$ is the midpoint of the {bisection points} of every pair in its bisector field. Since the bisector pairs of $Q$ are the pairs in the bisector field of $Q_0$ and have the same {bisection points}, {the midpoint of the two bisection points in each pair is the centroid of $Q$}.

Conversely, suppose that $Q$ is a triple-side quadrangle with triple side $T$. At least two non-triple sides $L_1,L_2$ are distinct, since otherwise $Q$ would not be a quadrangle. The pairs of opposite sides $\{T,L_1\}$ and $\{T,L_2\}$ are distinct bisector pairs containing $T$. Thus $T$ does not belong to a unique bisector pair.

(2) Suppose first that $Q$ is a trapezoid, and let $\{P,P'\}$ be any pair of parallel opposite sides. A triple-side quadrangle is not a trapezoid, since each pair of opposite sides contains the triple side and a non-triple side meeting it, so $Q$ has no triple side. The pair $\{P,P'\}$ belongs to $\mathcal B(Q)$ and hence is a bisector pair by Corollary~\ref{bisector field}. Let $M$ be any line parallel to $P$, and choose $M'$ so that $\{M,M'\}$ has the same midline as $\{P,P'\}$. The description of parallel degenerations in the proof of Proposition~\ref{pencil} shows that $\{M,M'\}\in\mathcal B(Q)$. Corollary~\ref{component bisectors} therefore shows that $M$ is a bisector of $Q$. Since $\{P,P'\}$ was arbitrary, this proves the forward implication and the final assertion.

Conversely, suppose that $Q$ has no triple side and let $\{M,M'\}$ be a parallel bisector pair. Take $Q_0$ to be the {complete quadrilateral} constructed in the proof of Theorem~\ref{weird}. If $M=M'$, the description of parallel degenerations in the proof of Proposition~\ref{pencil} yields a pair of distinct parallel lines in $\mathcal B(Q)$ with the same direction, so we may assume that $M\ne M'$. {The quadrangle $Q$ and quadrilateral $Q_0$ have the same bisectors,} so $M$ and $M'$ are distinct parallel bisectors of $Q_0$. By \cite[Proposition~3.2(3)]{OWQuad}, they are parallel to a pair of sides or diagonals of $Q_0$.

In the nondegenerate and double-side cases, the two pairs of opposite sides and the diagonal pair of $Q_0$ are the three pairs of opposite sides of $Q$, allowing repetitions. Thus $M$ and $M'$ are parallel to a pair of parallel opposite sides of $Q$. In the quadruple-side case, $Q$ is already a trapezoid, and the only parallel pair among the sides and diagonals of $Q_0$ is the pair $\BB=\{B,B'\}$ constructed parallel to the quadruple side $\ell$ since the other side pair $\AA=\{A,A'\}$ is nonparallel because $A$, $A'$, and $\ell$ form a nondegenerate triangle, while parallel diagonals would force the distinct side midpoints $A\cap\ell$ and $A'\cap\ell$ to coincide. Hence $M$ and $M'$ are parallel to $\ell$. This proves the converse and the assertion about every parallel bisector pair.

 (3) If $Q$ is a parallelogram,    $Q$ is nondegenerate by Proposition~\ref{unpack}, so its two   diagonals occur among its six sides and are distinct bisectors whose common {bisection point} is the centroid.
Conversely, suppose that two distinct bisectors $\ell$ and $\ell'$ have a common finite {bisection point} $m$. Lemma~\ref{triangle bf} shows that $Q$ has no triple side, so let $Q_0$ be the {complete quadrilateral} constructed in the proof of Theorem~\ref{weird}. The lines $\ell$ and $\ell'$ are distinct bisectors of $Q_0$ with common {bisection point} $m$. By \cite[Proposition~3.5]{OWQuad}, the vertices of $Q_0$ are the vertices of a parallelogram and $m$ is its centroid. Among the sides and diagonals of $Q_0$ are two parallel pairs in distinct directions, and these pairs belong to its bisector field by \cite[Theorem~6.4]{OWQuad}. Thus they are parallel bisector pairs of $Q$. By (2), $Q$ has pairs of parallel opposite sides in these two directions, so $Q$ is a parallelogram. Finally, by (1), $Q$ and $Q_0$ have the same centroid, and hence that $m$ is the centroid of $Q$. This proves~(3).
\end{proof}

\section{\texorpdfstring{{The locus of a quadrangle}}{The locus of a quadrangle}}\label{loci section}

{In this section, we associate a  conic to a quadrangle whose points can be described two ways. They are the centers of the central conics in the pencil of the quadrangle and the finite bisection points of the bisectors of the quadrangle.   For an application of this conic to moduli spaces of quadrangles, see \cite{OWManifold}, where an invariant associated to this conic is used as an organizing parameter for the moduli space.} One of our main purposes in this section is to determine to what extent the locus determines the entanglement class of a quadrangle, which we do in Theorem~\ref{locus determines}.

\begin{proposition}\label{center locus definition} \label{center locus equation definition}
\label{bisector locus definition}\label{center locus properties}
{Let $F,G\in\k[X,Y]$ be quadratic polynomials whose homogeneous degree-two parts are linearly independent, and set
$\mathcal P=\mathbb P(\operatorname{span}\{F,G\}).$
Then
$$J(F,G)=(\partial_XF)(\partial_YG)-(\partial_YF)(\partial_XG)$$
is a nonzero quadratic polynomial, the conic $[J(F,G)]$ is independent of the chosen basis of the pencil $\mathcal P$, and its affine $\k$-points are  the centers of the central conics in $\mathcal P$.}
\end{proposition}

\begin{proof} 
Write $\mathcal P=\mathbb P(\operatorname{span}\{F,G\})$, and write the homogeneous quadratic parts $F_2$ and $G_2$ of $F$ and $G$ as $$F_2=\alpha X^2+\beta XY+\gamma Y^2 \: {\text{ and }} \: G_2=\alpha'X^2+\beta'XY+\gamma'Y^2.$$ A direct calculation shows the homogeneous quadratic part of $J(F,G)$ is
\begin{eqnarray}
\label{jacobian homogeneous part}
J(F,G)_2=2(\alpha\beta'-\beta\alpha')X^2+4(\alpha\gamma'-\gamma\alpha')XY +2(\beta\gamma'-\gamma\beta')Y^2.
\end{eqnarray}
Since $\operatorname{char}(\k)\ne2$, this   vanishes exactly when the coefficient vectors $(\alpha,\beta,\gamma)$ and $(\alpha',\beta',\gamma')$ are scalar multiples of each other, but they are independent by hypothesis, so $J(F,G)$ has degree~$2$.

More generally,   when $H=aF+bG, K=dF+eG,$
and  $ae-bd\ne0$, we have $$J(H,K)=\det(a\nabla F+b\nabla G,d\nabla F+e\nabla G) =(ae-bd)J(F,G).$$
The polynomials $H=aF+bG$ and $K=dF+eG$ range over all   bases of $\operatorname{span}\{F,G\}$. Since $ae-bd\ne0$, the identity above shows that $J(H,K)$ is a nonzero scalar multiple of $J(F,G)$. Thus $[J(H,K)]=[J(F,G)]$, proving that {the conic $[J(F,G)]$} is independent of the chosen basis.  Set $W=\operatorname{span}\{F,G\}$, and for $m\in\k^2$ define
$$E_m:W\longrightarrow\k^2:H\longmapsto\nabla H(m).$$ In the basis $F,G$ of $W$ and the standard basis of $\k^2$, the determinant of $E_m$ is $J(F,G)(m)$. Moreover, $m$ is the center of a member of $\mathcal P$ exactly when $\ker E_m\ne0$. Since $W$ has dimension~$2$, this is equivalent to $\operatorname{rank}E_m<2$, and hence to $J(F,G)(m)=0$.
\end{proof}
 
{The idea of using the Jacobian to describe the centers of conics in a pencil goes back at least to Casey \cite[p.~171]{Casey}. We now show that, for the pencil of a quadrangle, the same conic {has as its points} the finite bisection points of its bisectors.}
{For nondegenerate real quadrangles inscribed in a circle, Sliep{\accent"14 c}evi\'{c} \cite[Theorem~4]{Sliepcevic} identifies the finite bisection points, called butterfly points there, with the centers of the central conics in the pencil.}

\begin{figure}[!htbp]
\centering
\includegraphics[width=\textwidth]{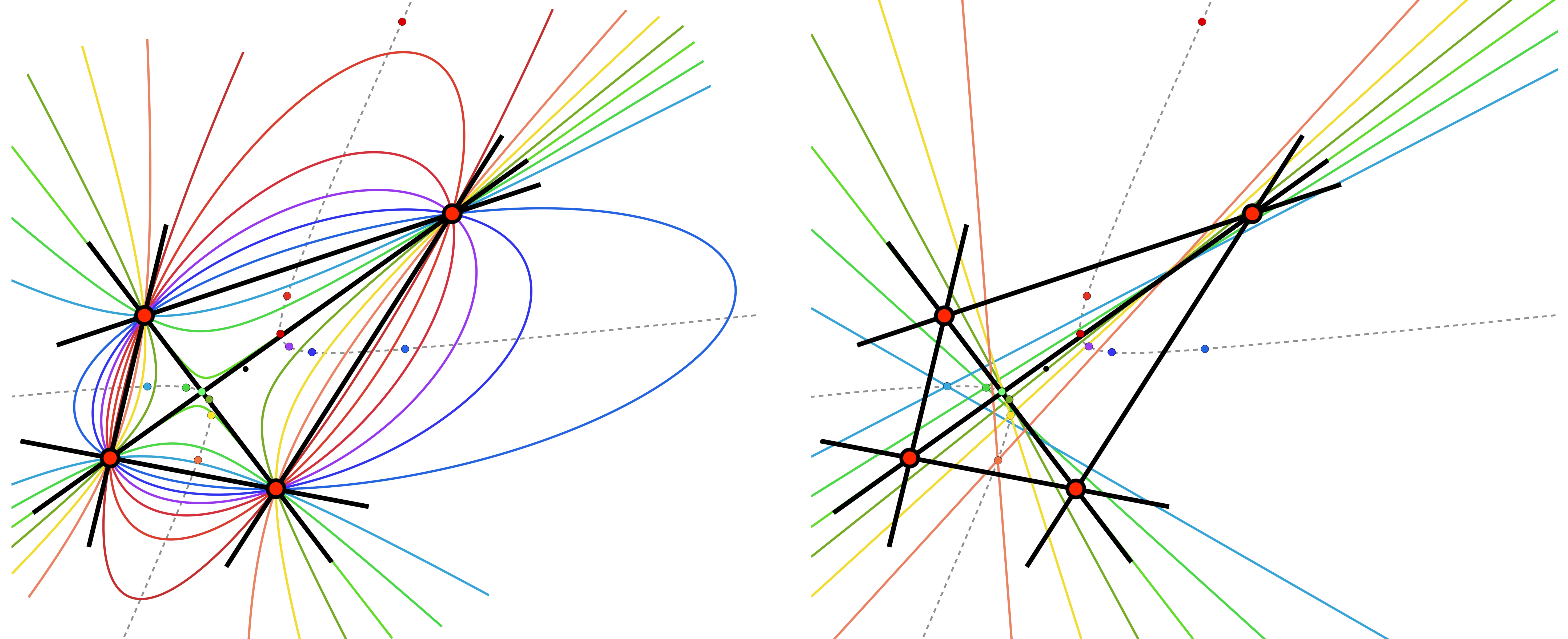}
\caption{Conics in the pencil of a quadrangle (left) and bisector pairs in its net (right), which are the reducible conics in the net. The colored points are  conic centers in both figures. The dashed curve is the common locus of conic centers and finite bisection points of the quadrangle's bisectors, as in Theorem~\ref{bisector locus}.}
\label{conics and bisectors figure}
\end{figure}

\begin{theorem} \label{bisector locus}
{Let $Q$ be a quadrangle, and let $F,G$ be a basis of $\mathcal V(Q)$. Then the finite bisection points of the bisectors of $Q$ are  the points of $V(J(F,G))$, and hence   the centers of the central conics in $\mathcal P(Q)$.}
\end{theorem}

\begin{proof}
 Let $W=\mathcal V(Q)$, and let $F,G$ be a basis of $W$.  By Proposition~\ref{pencil independent}, every nonzero element of $W$ has degree~$2$, so the homogeneous quadratic parts of $F$ and $G$ are linearly independent. Therefore, $F$ and $G$ satisfy the hypotheses of {Proposition~\ref{center locus properties}}. Define $E_m$ as in the proof of Proposition~\ref{center locus properties}. If $\ell$ is the line through $m$ with nonzero direction vector $w$, then Proposition~\ref{reflection} and Corollary~\ref{conic butterfly property} show that $\ell$ bisects $Q$ at $m$ exactly when $\nabla H(m)\mathbin{\cdot}w=0$ for every $H\in W$.
Such a vector $w$ exists exactly when $\operatorname{im}E_m$ is a proper subspace of $\k^2$, i.e., when $\operatorname{rank}E_m<2$. By Proposition~\ref{center locus properties}, this is the condition that {$m\in V(J(F,G))$, which is also the condition that $m$ be the center of a member of $\mathcal P(Q)$}.
\end{proof}

{The basis-independent conic}
$C_Q=[J(F,G)],$
{where $F,G$ is any basis of $\mathcal V(Q)$, is the \emph{locus conic} of $Q$, and $C_Q(\k)$ is the \emph{locus} of $Q$. By Theorem~\ref{bisector locus}, this locus is   the set of finite bisection points of the bisectors of $Q$.} For a nondegenerate quadrangle $Q$, the \emph{nine-point conic} of $Q$ is the affine conic whose projective realization passes through its six not-necessarily-distinct side midpoints and the three   intersections of its pairs of opposite sides. It is classically known that this conic is $C_Q$; see \cite[Corollary~5.3]{OWQuad} for a proof over fields of characteristic different from~$2$, and \cite{Del,Vac} for the historical development from Steiner through Trudi to Beltrami.  Thus the locus conic and the nine-point conic of a nondegenerate quadrangle are the same conic.

For the next theorem, an \emph{ellipse} is a nondegenerate conic with no $\k$-rational point at infinity, while a \emph{hyperbola} has two distinct $\k$-rational points at infinity and may be degenerate.
 
\begin{theorem}\label{locus}\label{locus equation}
The locus conic $C_Q$ of a quadrangle $Q$ is a degenerate hyperbola if and only if $Q$ is a trapezoid; it is a pair of parallel lines if and only if $Q$ is a triple-side quadrangle; otherwise it is a nondegenerate ellipse or a nondegenerate hyperbola. The centroid of $Q$ is a center of $C_Q$.

\end{theorem}
  
\begin{proof}
Suppose first that $Q$ has no triple side.  Let $Q_0$ be the {complete quadrilateral} constructed in the proof of Theorem~\ref{weird}, let $\AA$ and $\BB$ be its pairs of opposite sides, and let $F_{\AA}$ and $F_{\BB}$ be line-pair polynomials for them. {{By the construction in that proof,}
$$\mathcal V(Q)\oplus\k\mathbin{\cdot}1 =\operatorname{span}\{1,F_{\AA},F_{\BB}\}.$$
Since adding constants does not change gradients, the definition of $C_Q$ and Proposition~\ref{center locus properties} {imply} $C_Q=[J(F_{\AA},F_{\BB})].$
This is also the conic associated in \cite{OWQuad} with the bisection points of $Q_0$. {To verify that $C_Q$ is the conic defined in \cite[Theorem~5.1]{OWQuad}, use the reduction to standard form in the proof of \cite[Theorem~6.4]{OWQuad}. This reduction preserves the pencil up to affine change of coordinates. {The pencil generators in equation~(4) there are} $$F=\mu X^2+Y^2-4\mu hX-4kY+v,\quad G=\mu XY,$$
where $\mu\in\k^\times$, $v\in\k$, and $(h,k)$ is the centroid in these coordinates. Direct differentiation {shows that}
$$J(F,G)=-2\mu\bigl(Y^2-2kY-\mu(X^2-2hX)\bigr).$$ The expression in parentheses is the defining polynomial in the proof of \cite[Theorem~5.1]{OWQuad}. Since $\mu\ne0$, $J(F,G)$ defines that same conic. To pass back to the original coordinates, write an affine change of coordinates as $\phi(x)=Mx+b$, with $M$ invertible. {By the chain rule,} $\nabla(H\circ\phi)(x)=M^{\mathsf T}\nabla H(\phi(x))$ for every polynomial $H$. {On applying} this to $F$ and $G$ and taking determinants{, we have}
$$J(F\circ\phi,G\circ\phi)=\det(M)\,\bigl(J(F,G)\circ\phi\bigr).$$
Thus computing the Jacobian before or after changing coordinates produces the same conic. Together with basis independence, this shows that $J(F_{\AA},F_{\BB})$ is a nonzero scalar multiple of the defining polynomial for the conic associated to $Q_0$ in \cite[Theorem~5.1]{OWQuad}, with both polynomials expressed in the original coordinates.}}

The proof of Theorem~\ref{pairing structure}(1) also shows that $Q$ and $Q_0$ have the same centroid.  By \cite[Theorem~5.1]{OWQuad}, this conic is centered at the centroid of $Q_0$, so the centroid of $Q$ is a center of their common locus conic. This conic is either a nondegenerate ellipse or a hyperbola according to its points at infinity. To determine when the common locus conic is degenerate, first note that Theorem~\ref{pairing structure}(2) allows a repeated parallel bisector pair to be replaced by a pair of distinct parallel bisectors. By \cite[Proposition~3.2(3) and Theorem~5.2]{OWQuad}, the field contains such a pair exactly when $Q_0$ has parallel sides or diagonals, equivalently when the locus conic is degenerate. Theorem~\ref{pairing structure}(2) shows this is equivalent to $Q$ being a trapezoid. Thus the locus conic of a quadrangle with no triple side is a degenerate hyperbola exactly when the quadrangle is a trapezoid. Otherwise it is a nondegenerate ellipse or a nondegenerate hyperbola.
  
If $Q$ is a triple-side quadrangle, let $T$ be its triple side and $v$ the vertex off $T$. Each pair of opposite sides consists of $T$ and a non-triple side that meets $T$, so $Q$ is not a trapezoid. By Lemma~\ref{triangle bf}, each point on $T$ is a {bisection point} for $T$ as a bisector and hence is on the locus of $Q$, while the other lines parallel to $T$ have no {finite bisection points}. A bisector through $v$ that is not parallel to $T$ has as {its bisection point} the midpoint of $v$ and its intersection with $T$. These points all lie on  the line halfway between $v$ and $T$. Thus the locus of $Q$ is the union of this line and $T$, which are distinct because $v\notin T$. A defining polynomial of the locus conic $C_Q$ vanishes identically on both lines since each contains at least three $\k$-points. The product of their equations therefore divides this degree-two polynomial, so the locus conic of a triple-side quadrangle is a pair of distinct parallel lines.  Since the other three vertices lie on $T$, the centroid of $Q$ lies midway between these two lines and hence is a center of $C_Q$.
\end{proof} 

In general, two quadratic polynomials that are not scalar multiples of each other can have the same zero set in $\k^2$. The next theorem shows that such behavior does not occur for the locus conic of a quadrangle, and so the locus of $Q$ determines~$C_Q$. 
 
\begin{theorem}\label{locus determines equation}
The locus conic  of a quadrangle $Q$ is the unique conic $C$ such that $C(\k)$ is the locus of $Q$.
\end{theorem}

\begin{proof}
Let $Q$ be a quadrangle, let $F,G$ be a basis of $\mathcal V(Q)$, and set $J=J(F,G)$. By Proposition~\ref{center locus properties} and Theorem~\ref{bisector locus}, the zero set of $J$ is the locus of $Q$ and $C_Q=[J]$. Let $0\ne H\in\k[X,Y]$ have degree at most~$2$ and the same zero set as $J$. We show $[H] =[J]$. If $J$ is degenerate, then its zero set is the union of two distinct lines defined over $\k$, and since $|\k|\geq3$, the restriction of $H$ to either line has more zeroes than its degree and therefore vanishes identically. Thus both line equations divide $H$, so since $H$ is nonzero and has degree at most~$2$,  $[H] = [J].$

Now suppose that $J$ is nondegenerate. Write $H^h$ and $J^h$ for the degree-two homogenizations of $H$ and $J$. If $V(J)$ has more than four points, then the projective conics defined by $H^h$ and $J^h$ have more than four common points in $\P^2(\overline{\k})$. Bézout's theorem then implies $H^h$ and $J^h$ have a common component. Since $J^h$ is irreducible, the two homogenizations are proportional, with proportionality scalar in $\k$, and hence $[H]=[J]$. We may therefore assume that $V(J)$ has at most four points.

Every side midpoint of $Q$ belongs to $V(J)$ by Theorem~\ref{bisector locus}, so $V(J)$ is nonempty. Hence $\k$ must be finite, since a nondegenerate conic over an infinite field with a rational point has infinitely many rational points. Moreover, $Q$ cannot have four distinct vertices, since it would then have at least five distinct side midpoints. 
 Thus $Q$ has a repeated vertex. Since its locus conic is nondegenerate, Theorem~\ref{locus} shows that $Q$ has no triple side and is not a trapezoid. Thus Proposition~\ref{unpack}   implies that the vertices consist of one double vertex $p$ and two further points $q,r$ not collinear with $p$.

After an affine change of coordinates, we may assume $p=(0,0)$, $q=(2,0)$, and $r=(0,2)$. The four distinct side midpoints are then $(0,0), (1,0), (0,1), (1,1).$ Since $V(J)$ has four elements, the set of these four points is $V(J)$. Let  $s=|\k|$. A nondegenerate projective conic with a $\k$-point has ${s}+1$ projective $\k$-points, so Theorem~\ref{locus} shows that $|V(J)|$ is $s+1$ in the elliptic case and ${s}-1$ in the hyperbolic case. Since $|V(J)|=4$, the only possibilities are an ellipse over $\mathbb F_3$ and a hyperbola over $\mathbb F_5$.

Evaluation at the four specific points in $V(J)$ shows that every polynomial of degree at most~$2$ vanishing there has the form $$K=a(X^2-X)+c(Y^2-Y).$$
If the zero set of $K$ consists exactly of these four points, then $a,c\ne0$. After rescaling, assume $c=1$. Over $\mathbb F_3$, if $a=-1$, then $(2,2)$ is an additional zero, so $a=1$. Over $\mathbb F_5$, {$(2,3),(3,2),(3,3)$ are additional zeroes for $a=2,3,4$, respectively}, so again $a=1$.  Applying this conclusion to both $H$ and $J$, each is proportional to $X^2-X+Y^2-Y$. Hence $[H]=[J]$.
\end{proof}

{For real quadrangles, families sharing a nine-point conic were studied in the cyclic case by Sliep{\accent"14 c}evi\'{c} \cite[Theorem~5]{Sliepcevic} and, more generally, by {\accent"14 C}erin and Gianella \cite[Theorem~5]{CerinGianella}. Sliep{\accent"14 c}evi\'{c}'s construction fixes a pair of opposite sides and varies the radius of a circle with fixed center. The resulting quadrangles are (in our terminology) entangled since their pencils are generated by the fixed line pair and circle equations differing only by constants, {so their nets coincide after adjoining constants}.}

To conclude, we collect the characterizations of entanglement established so far, along with an additional characterization showing that the locus and one nonparallel bisector pair suffice to reconstruct the net of a quadrangle, and hence the entire bisector field.

\begin{theorem}\label{locus determines}
The following are equivalent for two quadrangles $Q_1$ and $Q_2$.
\begin{enumerate}
\item[$(1)$] The quadrangles are entangled.
\item[$(2)$] Every side of $Q_1$ bisects $Q_2$ at its side midpoint.
\item[$(3)$] Every side of $Q_2$ bisects $Q_1$ at its side midpoint.
\item[$(4)$] $\mathcal B(Q_1)=\mathcal B(Q_2)$.
\item[$(5)$] The quadrangles have the same locus, and they share a bisector pair whose lines are not parallel.
\end{enumerate}
\end{theorem}

\begin{proof}
The equivalence of~(1), (2), and~(3) follows from Corollary~\ref{entanglement bisection} and the symmetry of equality of nets, while Theorem~\ref{nice} shows~(1) and~(4) are equivalent. It remains to prove the equivalence of~(1) and~(5). Suppose first that~(1) holds. Choose bases $F_i,G_i$ of $\mathcal V(Q_i)$. Since the quadrangles have the same net, the two bases differ, modulo constants, by an invertible linear change of basis. Since constants do not affect gradients,  the calculation in the proof of Proposition~\ref{center locus properties} shows that $J(F_1,G_1)$ and $J(F_2,G_2)$ are nonzero scalar multiples. {Thus the quadrangles have the same locus conic and hence the same locus.} By the equivalence of~(1) and~(4), $\mathcal B(Q_1)=\mathcal B(Q_2)$.

We next show that $Q_1$ has a nonparallel pair of opposite sides. If $Q_1$ is degenerate, Proposition~\ref{unpack} {ensures that such a pair exists}. Suppose that $Q_1$ is nondegenerate. After an affine change of coordinates, we can assume $v_1=(0,0)$, $v_2=(1,0)$, and $v_3=(0,1)$. If either $\{L_{12},L_{34}\}$ or $\{L_{13},L_{24}\}$ is nonparallel, there is nothing to prove. Otherwise, the fact that  $L_{12}$ and $L_{34}$ are parallel implies $v_4 \in (0,1)+\k(1,0)$, while $L_{13}$ and $L_{24}$ being parallel implies  $v_4 \in (1,0)+\k(0,1)$. Hence $v_4=(1,1)$. The directions of $L_{14}$ and $L_{23}$ are then $(1,1)$ and $(-1,1)$. Thus $\{L_{14},L_{23}\}$ is nonparallel. Let $\AA$ be the resulting pair of nonparallel opposite sides. Since $\AA$ is a pair of opposite sides of $Q_1$, it belongs to $\mathcal B(Q_1)=\mathcal B(Q_2)$. Theorem~\ref{line-pair net members} therefore shows that $\AA$ is a bisector pair of both quadrangles. Thus~(5) holds.

Conversely, suppose~(5) holds. Let $\AA$ be a shared nonparallel bisector pair, and choose a line-pair polynomial $F$ for $\AA$. By Theorem~\ref{locus determines equation}, the two quadrangles $Q_1$ and $Q_2$ have the same locus conic. Choose polynomials $G_i$ such that
$\mathcal V(Q_i)+\k=\operatorname{span}\{1,F,G_i\}$
for $i=1,2$. There are scalars $\rho_{F,i},\rho_{G,i}\in\k$ such that $F-\rho_{F,i}$ and $G_i-\rho_{G,i}$ form a basis of $\mathcal V(Q_i)$. Constants do not affect the Jacobian, so {Proposition~\ref{center locus properties} shows that both $[J(F,G_1)]$ and $[J(F,G_2)]$ are the common locus conic}. After rescaling $G_2$, we may therefore assume $J(F,G_1)=J(F,G_2).$
After an affine change of coordinates that sends the lines in $\AA$ to the coordinate axes, we can assume $F=XY$. If $$H=aX^2+bXY+cY^2+rX+sY+t,$$ then $$J(XY,H)=2cY^2-2aX^2+sY-rX.$$
Consequently, $J(XY,H)=0$ if and only if $H\in\operatorname{span}\{1,XY\}.$
By linearity of the Jacobian   in its second argument, $J(F,G_1-G_2)=J(F,G_1)-J(F,G_2)=0.$
Thus $G_1-G_2\in\operatorname{span}\{1,F\}$. Hence $\mathcal V(Q_1)+\k=\mathcal V(Q_2)+\k$, and so $\mathcal N(Q_1)=\mathcal N(Q_2)$.  \end{proof}

\begin{example}
The requirement in Theorem~\ref{locus determines} that the shared pair be nonparallel cannot be omitted. For example, over $\mathbb R$, let $$\begin{aligned} Q_{\rm p}&=\{(-1,-1),(1,-1),(-1,1),(1,1)\},\\ Q_{\rm t}&=\{(-2,-1),(2,-1),(-1,1),(1,1)\}. \end{aligned}$$
The first four-point configuration is a parallelogram and the second is a nonparallelogram trapezoid, and they share the pair of parallel opposite sides defined by $F=Y^2-1$. For the respective pairs of opposite sides, take $$G_{\rm p}=X^2-1 \quad\text{and}\quad G_{\rm t}=(2X-Y+3)(2X+Y-3)=4X^2-Y^2+6Y-9.$$
{Their locus conics are represented by} $J(F,G_{\rm p})=-4XY$ and $J(F,G_{\rm t})=-16XY,$
so the two quadrangles have the same locus conic and hence the same locus. They are not entangled, however, since the fact that $G_{\rm t}$ has the nonzero linear term $6Y$ implies
$\operatorname{span}\{1,F,G_{\rm p}\} =\operatorname{span}\{1,X^2,Y^2\} \ne\operatorname{span}\{1,F,G_{\rm t}\}.$
\end{example}

{\it Funding.} The first author was supported by the National Science Foundation under grant DMS-2231414.

\smallskip

{\it Competing interests.} The authors declare no competing interests.
  
 \smallskip
 
 {\it Data availability.} No datasets were generated or analyzed during the current study.

\smallskip

{\it Technology and AI use.} During the preparation of this manuscript, the authors used ChatGPT   for checking and simplifying arguments. All mathematical   proofs and results however are the authors' own. Maple was used to assist with and check calculations.


\vspace{-.1in}


 

  \begin{thebibliography}{99}
  
  
  
  
  
  
  
  
  
  
  
\bibitem{BambergPenttila} J.~Bamberg and T.~Penttila, {\it Analytic Projective Geometry}, Cambridge University Press, Cambridge, 2023.

\bibitem{Casey} J.~Casey, {\it A Treatise on the Analytical Geometry of the Point, Line, Circle, and Conic Sections}, 2nd ed., revised and enlarged, Hodges, Figgis, \& Co., Dublin; Longmans, Green, \& Co., London, 1893.


\bibitem{CerinLines} Z.~{\accent"14 C}erin, Lines with the butterfly property, Math. Commun. 8 (2003), 35--41.


\bibitem{CerinGianella} Z.~{\accent"14 C}erin and G.~M.~Gianella, On improvements of the Butterfly theorem, Far East J. Math. Sci. 20(1) (2006), 69--85.

\bibitem{Coxeter} H.~S.~M.~Coxeter, {\it Projective Geometry}, 2nd ed., University of Toronto Press, Toronto, 1974.


\bibitem{Del} A.~Del Centina, Brianchon and Poncelet's joint memoir, the nine-point circle, and beyond, Arch. Hist. Exact Sci., 76(4) (2022), 363--390.

\bibitem{DragovicRadnovic} V.~Dragovi\'c and M.~Radnovi\'c, \emph{Poncelet Porisms and Beyond: Integrable Billiards, Hyperelliptic Jacobians and Pencils of Quadrics}, Frontiers in Mathematics, Birkh\"auser, Basel, 2011.

\bibitem{GonLop1} A.~Gonz\'alez and J.~L.~L\'opez-L\'opez, Compactness of spaces of convex and simple quadrilaterals, Bull. Aust. Math. Soc. 94(3) (2016), 507--521.

\bibitem{GonLop2} A.~Gonz\'alez and J.~L.~L\'opez-L\'opez, Spaces of special quadrilaterals, Bull. Aust. Math. Soc. 100(1) (2019), 155--167.

\bibitem{Kendall} D.~G.~Kendall, Shape manifolds, Procrustean metrics, and complex projective spaces, Bull. Lond. Math. Soc. 16(2) (1984), 81--121.



\bibitem{Loucka} P.~Lou{\accent"14 c}ka, Construction of special conics in pencils of conics using geometric algebra for conics, Adv. Appl. Clifford Algebr. 36 (2026), Article 20.


\bibitem{OWQuad} B. Olberding and E. A. Walker, Bisector fields of quadrilaterals, Beitr. Algebra Geom. 66 (2025), 135--157.


\bibitem{OWPencil} B. Olberding and E. A. Walker, Bisector fields and pencils of conics, Aequat. Math. 99 (2025), 351--376.

\bibitem{OWProj} B. Olberding and E. A. Walker, Geometry of lines asymptotic to a pencil of conics, Aequat. Math. 99 (2025),  2423--2456.
 


\bibitem{OWManifold} B. Olberding and E. A. Walker, The Nash manifold of four-point configurations  modulo similarity subgroups, submitted.

\bibitem{Sliepcevic} {A.~Sliep{\accent"14 c}evi\'{c}, A new generalization of the Butterfly theorem, J. Geom. Graph. 6(1) (2002), 61--68.}

\bibitem{Vac} M. Vaccaro, Historical origins of the nine-point conic---the contribution of Eugenio Beltrami, Hist. Math. 51 (2020), 26--48.

\end{thebibliography}
\end{document}